\documentclass[a4paper,11pt,oneside,notitlepage]{article}
\usepackage{times}
\usepackage{lmodern}
\usepackage[T1]{fontenc}
\usepackage[utf8]{inputenc}
\usepackage[english]{babel}
\usepackage{amssymb}
\usepackage{amsmath}
\usepackage{mathtools} % it allows to typeset formula in display style without any commands... \begin{dcases}
\usepackage{amsthm}
\numberwithin{equation}{section}
\usepackage{amsfonts, mathrsfs}
\usepackage{latexsym}
\usepackage[pdftex]{color, graphicx}
\usepackage{makeidx}
\usepackage{sidecap}
\usepackage{verbatim}
\usepackage{geometry}
\usepackage{subfig}
\usepackage{tikz,float}
\usetikzlibrary{shapes,arrows,shadows}
\usepackage[displaymath,mathlines, running]{lineno} % number of line
\usepackage{comment} % label interni
\usepackage[colorinlistoftodos]{todonotes}
\usepackage{esint} % per il simbolo della media integrale
\usepackage{pdfsync}
\definecolor{vg}{rgb}{0.0, 0.40, 0.15}
\newtheorem{thm}{Theorem}[section]
\newtheorem{prop}[thm]{Proposition}
\newtheorem{lemma}[thm]{Lemma}

\newtheorem{remark}[thm]{Remark}

\newcommand{\media}[1]{- \hskip-.9em \int_{#1}}

\def\R{\mathbb{R}}
\def\Z{\mathbb{Z}}
\def\N{\mathbb{N}}

\def\media
   {\mkern12mu\hbox{\vrule height4pt
    depth-3.2pt width5pt}
    \mkern-16.5mu\int}

 \newcommand{\mres}{\mathbin{\vrule height 1.6ex depth 0pt width
0.13ex\vrule height 0.13ex depth 0pt width 1.3ex}}

\title{Homogenization and 3D-2D dimension reduction of a functional on manifold valued Sobolev spaces}
\author{Michela Eleuteri, Luca Lussardi, Andrea Torricelli and Elvira Zappale}
\date{}

\AtEndDocument{\bigskip{\footnotesize%
  \textsc{ M.Eleuteri: Dipartimento di Scienze Fisiche Informatiche e Matematiche, Universit\`{a} degli studi di Modena e Reggio Emilia, via Campi 213/b, 41125 Modena, Italy} \par  
  \textit{E-mail address}: \texttt{michela.eleuteri@unimore.it}  \par
  \addvspace{\medskipamount}
  \textsc{L. Lussardi: Dipartimento di Scienze Matematiche G.L.\,Lagrange, Politecnico di Torino, C.so Duca degli Abruzzi 24, 10129 Torino, Italy} \par
  \textit{E-mail address}: \texttt{luca.lussardi@polito.it}
  \par
  \addvspace{\medskipamount}
  \textsc{A. Torricelli: Dipartimento di Scienze Matematiche G.L.\,Lagrange, Politecnico di Torino, C.so Duca degli Abruzzi 24, 10129 Torino, Italy} \par
  \textit{E-mail address}: \texttt{andrea.torricelli@polito.it}
  \par
  \addvspace{\medskipamount}
  \textsc{E. Zappale: Dipartimento di Scienze di base e Applicate per l'Ingegneria, Sapienza - Universit\`{a} di Roma,
via Antonio Scarpa, 16, 00161 Roma, Italy  and
CIMA, Universidade de \'Evora, Portugal} \par
  \textit{E-mail address}: \texttt{elvira.zappale@uniroma1.it}
}}

\begin{document}
\maketitle
\noindent 
{\bf Abstract.}
We study simultaneous homogenization and dimensional reduction of integral functionals for maps in manifold-valued Sobolev spaces. Due to the superlinear growth regime, we prove that the density of the $\Gamma$-limit is a tangential quasiconvex integrand represented by a cell formula.  

%\color{vg}
%\begin{itemize}
%.
%\item Sistemare introduzione
%\item Mettere in bibliografia lavori più recenti, e toglierne altri? ce ne sono comunque già 28...
%\item Dobiamo togliere le label dalle formule che non citiamo
%\item Dobbiamo rendere consistente la notazione della bibliografia% e citare Neukamm
%\item Controllare che valga la riduzione e homogenizzazione rispetto a tutte le dimensioni e confrontare con Neukamm
%\item Mettere a posto l'introduzione
%\item Tagliare le parti relative a BV
%\end{itemize}

\color{black}
\section{Introduction and main results}
The homogenization theory aims at describing the behavior of a model (either
partial differential equations or energy functional) with heterogeneous coefficients
that periodically oscillate on a small scale, say $h$. Indeed the main purpose consists of obtaining macroscopic properties of media with finely periodically distributed micro-structure, rigorously deriving these properties by means of a limiting procedure  as the fine-scale $h$ converges to zero. Many approaches have been developed in the last century: for instance
asymptotic expansion methods (e.g. see A. Bensoussan, J.L. Lions and G. Papanicolaou \cite{BLP78}, E. Sanchez-Palencia \cite{SP80}) or the H-convergence methods due to
F. Murat and L. Tartar \cite{Tar77, Tar09, FMT09} or the two-scale convergence \cite{Ngu89, All92}, more recently re-casted in terms of a fixed functional space, i.e. within the theory of periodic unfolding (see \cite{CDG02, Dam05, Vis06, Vis07, CDGbook}).

From the variational stand-point one 
 is interested in the asymptotic behavior of minimizers
of energy functionals
 depending on this small parameter $h$, and a crucial tool in this framework is $\Gamma$-convergence, see \cite{DMbook}.

Another area of research in elasticity and micromagnetics is the derivation of
lower dimensional theories — such as membrane, plate, string and rod models —
from three-dimensional samples, i.e. one is interested in detecting a reduced model asymptotically departing from a slender body, 
letting the geometry of the body become singular.
Again a rigorous approach is $\Gamma$-convergence and many results have been obtained in this context, after the pioneering papers \cite{ABP91}, \cite{AL01}, \cite{LDR95}, \cite{GJ}, \cite{BZ}, \cite{GH}, \cite{CCG}, among a wider literature,
in the elastic and micromagnetic case respectively.  In this paper we consider the slender domain approaching the reduced one as the heterogeneity becomes finely and finely distributed, i.e. we consider homogenization and dimensional reduction to happen simultaneously. This analysis has been performed in the realm of nonlinear heterogeneous thin structures and composites (\cite{BFF}, \cite{FIOM14}, \cite{KK16}, \cite{BV17}, \cite{FG23}, \cite{CGO24}) or with the two-scale convergence technique and for plates and rods (\cite{Neukamm}). On the other hand, the same procedure has not been taken into account in the constrained setting, suitable to model liquid crystals, magnetostrictive or ferromagnetic materials, etc.
\color{black}

% This theory covers a wide range of applications
% that goes from the study of the properties of composites to optimal design.
In the current paper we assume that the domain is an inhomogeneous cylinder, whose microstructure is assumed to be distributed with periodicity within the material described by the small parameter $h$ comparable with the height of the domain. The equilibrium configurations are detected as minimizers of an
integral functional of the form 
%\begin{equation}
%\label{funz-modello}
\[
\int_{\omega\times (-\frac{h}{2},\frac{h}{2})} f\left ( \frac{x}{h}, \nabla u\right ) \, dx \qquad u: \omega\times \left (-\frac{h}{2},\frac{h}{2} \right ) \rightarrow \mathbb{R}^3,
\]
%\end{equation}
under suitable boundary conditions, where $\omega \subset \mathbb{R}^2$ is a bounded open set, and $f: \mathbb{R}^3 \times \mathbb{R}^{3 \times 3}\rightarrow \mathbb{R}$
is a periodic integrand with respect to the first variable, and $u$ is a manifold-valued Sobolev field, that will be specialized in the sequel. 
%The asymptotic behavior of minimizers of such energies can be performed via the $\Gamma-$convergence tools (for more details see for instance \color{orange} una qualche monografia sulla Gamma-convergenza). 
\color{black}
%Indeed, in many applications, the admissible fields have to satisfy additional constraints, therefore it may be interesting to understand the behaviour of integral functionals of the type \eqref{funz-modello} under this additional requirement. 
% In the framework of Sobolev spaces, we can quote for instance \cite{Dacorogna-Fonseca-Maly-Trivisa}, \cite{BM}, where the relaxed energy is obtained by replacing the integrand by its tangential quasiconvexification which is the analogue of the quasiconvex envelope in the non constrained case. 
Due to the many applications, it is worth to recall that solely the homogenization of integral functionals depending on $x$ and $\nabla u$ and defined on manifold-valued Sobolev fields has been studied by Babadjian and Millot in \cite{BM} for $u\in W^{1,p}$ and in \cite{BMbv} for $u\in W^{1,1}$, we also refer to \cite{P} for other related models.
Analogously the dimensional reduction of micromagnetic and ferromagnetic energy has been studied, in several contexts, we recall \cite{AL01, BZ07, GH, GH2, CCG, DKMO} among the others.

The simultaneous homogenization and dimension reduction of an integral functional defined on real valued Sobolev functions has been studied by Braides, Fonseca and Francfort in \cite{BFF} in the case $p>1$, while the case $p=1$ can be covered by the Global Method \cite{BFM}. 

The main novelty of our contribution is to apply both homogenization and dimension reduction simultaneously and to assume the functional defined on the space of manifold-valued Sobolev functions for $p>1.$ In a subsequent contribution \cite{LTZ} we will deal with the case $p=1$.\\

For the sake of exposition, we focus on the model $3D-2D$ but our analysis could be extended to other dimensions, i.e. to the framework $ND- (N-d)D$.
\noindent
Given a Carathéodory function $f:\R^3 \times\R^{3\times 3}\to \mathbb{R}$ and $1 <p<+\infty$ we consider the functional
\[
\frac{1}{h}\int_{\omega_{,h}} f \left (\frac{x}{h}, \nabla u\right ) \, dx, \qquad u \in W^{1,p}(\omega_{,h}; \mathcal{M}),
\]
with $\omega_{,h} := \omega \times \left (- \frac{h}{2}, \frac{h}{2} \right ),$ $h>0,$  $\omega \subset \mathbb{R}^2$ open and bounded, and $\mathcal{M}$ a smooth submanifold of $\mathbb{R}^3$ without boundary. In particular, we assume that $f$ has the following properties:
\begin{itemize}
    \item [(H1)] $f(\cdot, x_3, \xi)$ is 1-periodic, i.e. for every $(x_{\alpha}, x_3) \in \mathbb{R}^3$ and $\xi \in \mathbb{R}^{3 \times 3}$ it holds
    \begin{equation*}
    \label{H1}
        f(x_{\alpha} + \mathbf{e}_i, x_3, \xi) = f(x_{\alpha}, x_3, \xi), \qquad \forall i = 1,2
    \end{equation*}
    where $\{\mathbf{e}_1, \mathbf{e}_2\}$ is the canonical basis of $\mathbb{R}^2;$

    \item [(H2)] $f$ has $p$-growth: there exists $\alpha, \beta > 0$ and $1 < p < + \infty$ such that
    %\begin{equation}
    %\label{H2}
    \[
        \alpha |\xi|^p \le \, f(x, \xi) \le \, \beta (1 + |\xi|^p),
        \]
   % \end{equation}
    for a.e. $x\in\R^3$ and for every $\xi\in\R^{3\times 3}$.
\end{itemize}

We define the functional $\tilde{I}^h: L^p(\omega_{,h}; \mathcal{M}) \rightarrow \overline{\mathbb{R}}$
\[
\tilde{I}^h(u) := \left \{
\begin{array}{lll}
\!\!\!\!\!\! & \displaystyle \frac{1}{h}\int_{\omega_{,h}} f \left (\frac{x}{h}, \nabla u\right ) \, dx \qquad & \textnormal{if $u \in W^{1,p}(\omega_{,h}; \mathcal{M})$}\\[4mm]
\!\!\!\!\!\! & + \infty &\textnormal{elsewhere.}
\end{array}
\right.
\]
The study of the $\Gamma-$limit of $\tilde{I}^h$ is equivalent to the study of the $\Gamma-$limit of the rescaled functional $I^h$ defined as
\begin{equation}
\label{functional}
    I^h(u) := \left \{
\begin{array}{lll}
\!\!\!\!\!\! & \displaystyle \int_{\Omega} f \left (\frac{x_{\alpha}}{h}, x_3, \nabla_h u\right ) \, dx \qquad & \textnormal{if $u \in W^{1,p}(\Omega; \mathcal{M})$}\\[4mm]
\!\!\!\!\!\! & + \infty &\textnormal{elsewhere}
\end{array}
\right.
\end{equation}
with $\Omega := \omega \times \left (- \frac{1}{2}, \frac{1}{2} \right ) = \omega_{,1}$ and $\nabla_h:=[\frac{\partial}{\partial x_1}, \frac{\partial}{\partial x_2}, \frac{1}{h}\frac{\partial}{\partial x_3}]$.
%For the sake of notation we denote $f_h(x, \xi):=f\left (\frac{x_{\alpha}}{h}, h x_3, \xi\right )$ for every $x\in \R^3$ and $\xi \in \R^{3\times 3}$. It clear that $f_h$ has the very same properties of $f$.
From here onward, we denote $\nabla_\alpha:=[\frac{\partial}{\partial x_1}| \frac{\partial}{\partial x_2}]$ and $\nabla_3:=\frac{\partial}{\partial x_3}$, so that $\nabla_h=[\nabla_\alpha| \frac{1}{h}\nabla_3]$. Moreover, we denote with $\xi_\alpha$ an element of $\R^{3\times 2}$ and with $\xi$ an element of $\R^{3\times 3}.$\\

\noindent
Our main result is the following.

\begin{thm}
\label{casopm1}
    Assume that $\mathcal{M}$ is a connected smooth manifold of $\mathbb{R}^3$ without boundary and let $f:\R^3 \times\R^{3\times 3}\to \mathbb{R}$ be a Carathéodory function satisfying {\rm (H1)} and {\rm (H2)} with $1<p<+\infty$. Then, the $\Gamma-$limit of $I^h$ as $h \rightarrow 0$ with respect to the strong $L^p$-topology is the functional $I: L^p(\omega; \mathcal{M}) \rightarrow \overline{\mathbb{R}}$ given by 
\begin{equation}
\label{candidatepm1}
    I(u) = \left \{
\begin{array}{lll}
\!\!\!\!\!\! & \displaystyle\int_{\omega} T f_{\rm hom}^0(u, \nabla_\alpha u) \, dx_\alpha, & \textnormal{if $u \in W^{1,p}(\omega; \mathcal{M})$}\\[4mm]
\!\!\!\!\!\! & + \infty &\textnormal{elsewhere,}
\end{array}
\right.
\end{equation}
with $T f^0_{\rm hom}: \mathbb{R}^3 \times \mathbb{R}^{3 \times 2} \rightarrow \mathbb{R}$ defined as
\begin{align}
\label{Tfhom}
&Tf^0_{\rm hom}(s,\xi_\alpha) :=  \liminf_{t \rightarrow + \infty} \inf_{\varphi} \Bigg \{\frac{1}{t^2} \int_{(tQ')_{,1}} f(x_{\alpha}, x_3, \xi_\alpha + \nabla_{\alpha} \varphi | \nabla_3 \varphi) \, d x_{\alpha} d x_3: \nonumber\\
& \, \varphi \in W^{1,\infty}((tQ')_{,1}; T_s (\mathcal{M})), \,\, \varphi(x_{\alpha}, x_3) = 0\,\,\, \textnormal{for every $(x_{\alpha}, x_3) \in \partial (tQ') \times \left(-\frac{1}{2},\frac{1}{2}\right)$} \Bigg \},
\end{align}
where $Q':=\left(-\frac{1}{2},\frac{1}{2}\right)^2,$ $(tQ')_{,1}:=tQ'\times \left(-\frac{1}{2},\frac{1}{2}\right), $ and $T_s (\mathcal{M})$ denotes the tangent space to $\mathcal{M}$ in $s$.
\end{thm}

We observe that, in view of $(H2)$, the same result could be achieved replacing the strong convergence in $L^p(\omega;\mathcal M)$ by the weak convergence in $W^{1,p}(\omega;\mathcal M)$.

The paper is organized as follows: Section \ref{NP} is devoted to fix notation and to provide preliminary results, while Section \ref{Sobolev} is devoted to the proof of Theorem \ref{casopm1}.

The $\Gamma-$liminf inequality follows by the localization and blow-up methods, while the main difficulty with the $\Gamma-$limsup inequality is due to the fact that we need to construct a manifold-valued recovery sequence. A key tool in our result is the {\it tangential quasiconvexification} introduced in \cite{Dacorogna-Fonseca-Maly-Trivisa}.
Our integral representation result holds in $W^{1,p}(\omega;\mathcal M)$ when $p>1$, but we observe that \eqref{candidatepm1} is still valid, and can be obtained with the same (even easier) techniques, when $p=1$ and $L^p(\omega;\mathcal M)$ strong convergence is replaced by $W^{1,1}(\omega;\mathcal M)$, besides in this latter setting there is a lack of compactness. On the other hand, we emphasize that when $p=1$  and $u\in W^{1,1}(\omega;\mathcal M)$ \eqref{candidatepm1} represents also the $\Gamma-L^1(\omega;\mathcal M)$ strong limit of \eqref{functional}, when $h \to 0$. For this result we refer to \cite{LTZ}.

\section{Notation and preliminaries}\label{NP}
This section is devoted to fix notation, recall previous results that will be useful in the sequel and establish properties of the energy densities appearing in the main result.
\color{black}

In what follows $\Omega:=\omega\times (-\frac{1}{2}, \frac{1}{2})$, $\omega \subset \mathbb{R}^2$ is open and bounded, and $\mathcal{M}$ is a smooth submanifold of $\R^3$ without boundary, further assumptions on $\mathcal{M}$ will be stated explicitly if needed. Given $s\in\mathcal{M}$, we write $T_s(\mathcal{M})$ for the tangent space to $\mathcal{M}$ in $s$.
%Given $a, b \in \mathcal{M},$ we introduce the family of geodesic curves between $a$ and $b$ by
%\begin{align}
%\label{geodesic}
%& \mathcal{G}(a,b) :=  \Big\{ \gamma \in \mathcal{C}^{\infty}(\mathbb{R}; \mathcal{M}): \gamma(t) = a, \,\,\, {\rm if} \,\,\,t \ge 1/2, \nonumber\\
%&  \gamma(t) = b, \,\,\, {\rm if} \,\,\,t \le -1/2, \,\,\, \int_{\mathbb{R}} |\dot{\gamma}| \, dt = {\bf d}_{\mathcal{M}}(a,b) \Big \},
%\end{align}
%where ${\bf d}_{\mathcal{M}}$ denotes the geodesic distance on $\mathcal{M}.$
%\\
We denote by $\mathcal{A}(\Omega)$ the family of all open subsets of $\Omega$ and with $\mathcal{A}(\omega)$ the family of all open subset of $\omega$. We write $B^k(s,r)$ for the closed ball in $\R^k$, $k\in\N$, of center $s\in\R^k$ and radius $r>0$. Moreover, we denote by $Q$ the cube $(-\frac{1}{2}, \frac{1}{2})^3$ and with $Q(x_0,\rho)$ the rescaled and translated cube $x_0+\rho Q$, with $x_0\in\R^3$, $\rho>0$. In a similar way, given $\nu \in \mathbb{S}^1,$ $Q_{\nu}$ stands for the open unit square in $\mathbb{R}^2$ centered at the origin, with the two sides orthogonal to $\nu;$ we set $Q_{\nu}(x_0, \rho) := x_0 + \rho Q_{\nu}.$ 
\\
Furthermore, given a set $A\subset\omega$ we denote by $A_{,h}$, with $h>0$, the set $A\times (-\frac{h}{2}, \frac{h}{2})$ and so $A_{,1}:=A\times(-\frac{1}{2}, \frac{1}{2})$. It follows that, in particular, $\Omega=\omega_{,1}$. We also denote by $Q'$ the square $(-\frac{1}{2}, \frac{1}{2})^2$, so $Q'_{,h}=(-\frac{1}{2}, \frac{1}{2})^2\times (-\frac{h}{2}, \frac{h}{2})$ for $h>0$, while $Q'_{,1}=(-\frac{1}{2}, \frac{1}{2})^3=Q$ and $Q_{\nu,1} = Q_{\nu} \times (-\frac{1}{2}, \frac{1}{2}).$ $\mathcal{M}(\Omega)$ is the space of real valued Radon measure in $\Omega$ with finite total variation,  $\mathcal{L}^k$, is the $k$-dimensional Lebesgue measure, with $k\in\N$. Finally, given $\lambda,\mu\in\mathcal{M}(\Omega)$ we denote by $\frac{d\lambda}{d\mu}$ the Radon-Nykod\'ym derivative of $\lambda$ with respect to $\mu$. By a generalization of Besicovitch Differentiation Theorem, see for instance \cite[Proposition 2.2]{Ambrosio-Dal Maso}, there exists a Borel set $E$ such that $\mu(E)=0$ and $\frac{d\lambda}{d\mu}(x)=0$ for every $x\in {\rm supp}\, \lambda \setminus E.$\\

\noindent
Given $s\in\mathcal{M}$, we consider the orthogonal projection %\begin{equation}\label{orthP}
\[
P_s: \R^3\to T_s(\mathcal{M}),
\]
%\end{equation} 
and we define the function 
$\mathbf{P}_s: \R^{3\times 3}\to [T_s(\mathcal{M})]^3$ as
%\begin{equation}\label{3dorthP}
\[
\mathbf P_s(\xi):=\left[P_s(\xi_1)| P_s(\xi_2)| P_s(\xi_3)\right],
\]
%\end{equation}
for every $\xi=[\xi_1|\xi_2|\xi_3]\in\R^{3\times 3}.$ 

For a Carathéodory function $f:\mathbb R^3 \times \mathbb R^{3 \times 3}\to \mathbb R$, we set
\begin{equation}
\label{perturbedf}
    \Bar{f}(x,s,\xi):=f(x,\mathbf{P}_s(\xi))+|\xi-\mathbf{P}_s(\xi)|^p.
\end{equation}
This function will play a crucial role in our subsequent analysis, since it will appear in the formulas to detect our limiting energy densities.
 By construction the function $\Bar{f}:\R^3\times\mathcal M\times\R^{3\times 3}\to\R$ is Carathéodory, i.e. it is measurable with respect to the first variable and continuous with respect to the last two variables. Moreover, if conditions (H1) and (H2) for $1< p<+\infty$ are satisfied then $\Bar{f}$ is $1$-periodic in the first variable and satisfies $p$-growth and $p$-coercivity conditions, i.e. there exists $C>0$ such that
 \begin{equation}\label{pgrowth}
 \frac{1}{C}|\xi|^p\le\Bar{f}(x,s,\xi)\le C(1+|\xi|^p),
 \end{equation}
for every $(s,\xi)\in \mathcal M\times\R^{3\times 3}$ and for a.e. $x\in\R^3$.

Following \cite[Proposition 2.1]{BM}, we characterize the density $Tf_{\rm hom}^0$ of the $\Gamma$-limit and we will prove some of its properties.

\begin{prop}
    \label{characterization}
   Let $1\le p < +\infty$,  for every  Carathéodory function $f:\R^3\times\R^{3\times 3}\to\R$ satisfying conditions {\rm (H1)} and {\rm (H2)} the following properties hold.
    \begin{itemize}
        \item[{\rm (i)}] For every $s\in\mathcal{M}$ and $\xi_\alpha\in[T_s(\mathcal{M})]^2$
%         \begin{align}
%             &Tf^0_{\rm hom}(s,\xi):= \label{Tf0hom}\\
%             &=  \lim_{t \rightarrow + \infty} \inf_{\varphi} \Bigg \{\frac{1}{t^2} \int_{(tQ')_{,1}} f(x_{\alpha}, x_3, \xi + \nabla_{\alpha} \varphi | \nabla_3 \varphi) \, d x_{\alpha} d x_3: \nonumber\\
% & \, \varphi \in W^{1,\infty}((tQ')_{,1}; T_s (\mathcal{M})), \,\, \varphi(x_{\alpha}, x_3) = 0\,\,\, \textnormal{for every $(x_{\alpha}, x_3) \in \partial (tQ') \times \left(-\frac{1}{2},\frac{1}{2}\right)$} \Bigg \}, \nonumber
%         \end{align}
%         and
        \begin{equation}
        \label{(2.3)}
        Tf_{\rm hom}^0(s, \xi_\alpha)=\Bar{f}_{\rm hom}^0(s,\xi_\alpha),
        \end{equation}
        where $Tf^0_{\rm hom}$ is defined as in \eqref{Tfhom} and
        \begin{align}
            &\Bar{f}^0_{\rm hom}(s,\xi_\alpha) =  \lim_{t \rightarrow + \infty} \inf_{\varphi} \Bigg \{\frac{1}{t^2} \int_{(tQ')_{,1}} \Bar f(x_{\alpha}, x_3,s, \xi_\alpha + \nabla_{\alpha} \varphi | \nabla_3 \varphi) \, d x_{\alpha} d x_3: \label{f0hom}\\
& \, \varphi \in W^{1,\infty}((tQ')_{,1}; \R^3), \,\, \varphi(x_{\alpha}, x_3) = 0\,\,\, \textnormal{for every $(x_{\alpha}, x_3) \in \partial (Q') \times \left(-\frac{1}{2},\frac{1}{2}\right)$} \Bigg \} \nonumber
        \end{align}
        
        \item[{\rm (ii)}] The function $Tf_{\rm hom}^0$ is tangentially quasiconvex in the second variable, i.e.
        \[
        Tf_{\rm hom}^0(s,\xi_\alpha)\le\int_{Q'} Tf_{\rm hom}^0(s, \xi_\alpha+ \nabla_\alpha \psi)dx_\alpha,
        \]
        for every $s\in\mathcal{M}, \xi_\alpha\in [T_s(\mathcal{M})]^2$, and $\psi\in W^{1,\infty}_0(Q'; T_s (\mathcal{M}))$. In particular, $Tf_{\rm hom}^0(s, \cdot)$ is rank one convex.
        \item[{\rm (iii)}] $Tf_{\rm hom}^0$ is uniformly $p$-coercive and has uniform $p$-growth in the second variable (i.e. it satisfies inequalities as in \eqref{pgrowth}, uniformly with respect to $s$). Moreover, there exists $C>0$ such that for every $s\in\mathcal{M}$ and $\xi_\alpha, \xi_\alpha'\in [T_s(\mathcal{M})]^2$  
        \begin{equation}
        \label{(2.5)}
        |Tf_{\rm hom}^0(s, \xi_\alpha) - Tf_{\rm hom}^0(s, \xi_\alpha')|\le C|\xi_\alpha-\xi_\alpha'|(1+|\xi_\alpha|^{p-1}+|\xi_\alpha'|^{p-1}).
        \end{equation}

      %  \item[{\rm (iv)}] Finally, if hypothesis (H4) is also satisfied, then \begin{equation}
       % \label{BMBV3.11}
        %    |Tf^0_{\rm hom}(s,\xi_\alpha)-Tf^{0,\infty}_{\rm hom}(s,\xi_\alpha)|\le C(1+|\xi_\alpha|^{1-q}),
        %\end{equation} 
        %for some positive constant $C$ and for every $s\in\mathcal{M}$ and $\xi_\alpha \in [T_s(\mathcal{M})]^2$, where $Tf^{0,\infty}_{\rm hom}$ is the function appearing in \eqref{ohominfty}.
    \end{itemize}
\end{prop}

Before proving our statement, it is worth observing that for every $s \in \mathcal M$, $\Bar{f}^0_{\rm hom}(s,\cdot)$ in \eqref{f0hom}, is the $3D$-$2D$ homogenized energy density appearing in the dimension reduction problems in the unconstrained setting and it has been introduced in \cite{BFF}.
\color{black}
\begin{proof}
We start from {\rm (i)}. Fix $s \in \mathcal{M}$ and $\xi_\alpha \in [T_s(\mathcal{M})]^2.$ For any $t > 0,$ we introduce
\[
\begin{aligned}
&T f_t (s, \xi_\alpha) := \inf_{\varphi} \left \{\media_{(tQ')_{,1}} f(y, \xi_\alpha + \nabla_\alpha \varphi|\nabla_3\varphi) d y:\right.\\
&\left.\varphi \in W^{1,\infty}((tQ)_{,1};T_s(\mathcal{M})),\,\varphi(x_{\alpha}, x_3) = 0\,\,\, \textnormal{for every $(x_{\alpha}, x_3) \in \partial (tQ') \times \left(-\frac{1}{2},\frac{1}{2}\right)$}\right \},
\end{aligned}
\]
and 
\begin{align}
&\bar{f}_t (s, \xi_\alpha) := \inf_{\varphi} \left \{ \media_{(tQ')_{,1}} \bar{f}(y, s,\xi_\alpha + \nabla_\alpha \varphi|\nabla_3 \varphi) d y:\right.\label{ftbardef}\\
&\left.\,\, \varphi \in W^{1,\infty}((tQ')_{,1}; \R^3),\,\varphi(x_{\alpha}, x_3) = 0\,\,\,\textnormal{for every $(x_{\alpha}, x_3) \in \partial (tQ') \times \left(-\frac{1}{2},\frac{1}{2}\right)$}\right \}.\nonumber
\end{align}

Therefore, from \eqref{Tfhom} and \eqref{f0hom}, we have that $Tf_{\rm hom}^0(s,\xi_\alpha)=\liminf_{t\to \infty} Tf_t(s,\xi_\alpha)$ and $\bar f_{\rm hom}^0(s,\xi_\alpha)=\lim_{t\to \infty}\bar{f}_t(s,\xi_\alpha)$
 provided the latter one exists.

Following a classical argument (see for instance \cite[Proposition 14.4]{BDF} and \cite[Proposition 4.1]{CRZ11}) which, for every $\tau>t>0$ allows to cover $\tau Q'$ by non overlapping squares of side length $t$ and construct suitable test functions $\varphi_\tau$  for the definition of $\overline f_\tau(s,\xi_\alpha)$ in \eqref{ftbardef},  the existence of the following limit
\[
\lim_{t \rightarrow + \infty} \bar{f}_t(s, \xi_\alpha) \qquad \textnormal{for every $s \in \mathcal{M}$ and $\xi_\alpha \in [T_s(\mathcal{M})]^2$},
\]
is granted.
Therefore, in order to conclude the proof, it is enough to show that $T f_t (s, \xi_\alpha) = \bar{f}_t(s, \xi_\alpha)$ for every $t > 0$. For any $\varphi \in W^{1,\infty}((tQ)_{,1};T_s(\mathcal{M}))$ with $\varphi(x_{\alpha}, x_3) = 0$ for every $(x_{\alpha}, x_3) \in \partial (tQ') \times \left(-\frac{1}{2},\frac{1}{2}\right)$ we have
\[
\bar{f}_t(s, \xi_\alpha) \le \, \media_{(tQ')_{,1}} \bar{f}(y, s, \xi_\alpha + \nabla_\alpha \varphi|\nabla_3 \varphi) \, dy = \media_{(tQ')_{,1}} f(y,\xi_\alpha + \nabla_\alpha \varphi|\nabla_3 \varphi) \, dy,
\]
since $(\xi_\alpha + \nabla_\alpha \varphi|\nabla_3 \varphi) \in [T_s(\mathcal{M})]^3$ for a.e. $y \in (tQ')_{,1}.$ By taking the infimum over $\varphi$ in the right hand side of the previous inequality, we obtain that
\[
\bar{f}_t (s, \xi_\alpha) \le \, T f_t(s, \xi_\alpha).
\]
With the aim to prove the converse inequality, consider $\psi \in W^{1,\infty} ((tQ')_{,1}; \mathbb{R}^3))$ such that $\psi(x_{\alpha}, x_3) = 0$ for every $(x_{\alpha}, x_3) \in \partial (tQ') \times \left(-\frac{1}{2},\frac{1}{2}\right)$, and set 
\[
\tilde{\psi} = P_s(\psi).
\]
It is not difficult to see that
\[
\tilde{\psi} \in W^{1, \infty}((tQ')_{,1}; T_s(\mathcal{M})) \qquad \textnormal{and} \qquad \nabla \tilde{\psi} = \mathbf{P}_s(\nabla \psi) \,\, \textnormal{a.e. in $(tQ')_{,1},$}
\]
with $\tilde \psi(x_{\alpha}, x_3) = 0$ for every $(x_{\alpha}, x_3) \in \partial (tQ') \times \left(-\frac{1}{2},\frac{1}{2}\right)$. Thus we get
\[
\begin{aligned}
T f_t(s,\xi_\alpha) &\le \, \media_{(tQ')_{,1}} f(y, \xi_\alpha + \nabla_\alpha \tilde{\psi}|\nabla_3\tilde{\psi}) \, dy\\
&= \media_{(tQ')_{,1}} f(y,\mathbf{P}_s(\xi_\alpha + \nabla_\alpha \psi|\nabla_3 \psi)) \, dy\\
&\le \, \media_{(tQ')_{,1}} \bar{f}(y, s,\xi_\alpha + \nabla_\alpha \psi|\nabla_3 \psi) \, dy.
\end{aligned}
\]
Then, the thesis follows by taking the infimum over $\psi$ in the right hand side of the inequality.
\\
In order to prove {\rm (ii)}, we observe that the classical results (see for instance \cite[Theorem 4.2]{BFF}) yield that $\bar{f}^0_{\rm hom}(s, \cdot)$ is a quasiconvex function for every $s \in \mathcal{M};$ consequently, for any $s \in \mathcal{M}, \xi_\alpha \in [T_s(\mathcal{M})]^2$ and $\varphi \in W^{1, \infty}_0(Q', T_s(\mathcal{M})),$ it holds
\[
T f^0_{\rm hom} (s, \xi_\alpha) = \bar{f}^0_{\rm hom}(s, \xi_\alpha) \le \, \int_{Q'} \bar{f}^0_{\rm hom}(s, \xi_\alpha + \nabla_\alpha \varphi) \, dy_\alpha = \int_{Q'} T f^0_{\rm hom} (s, \xi_\alpha + \nabla_\alpha \varphi) dy_\alpha.
\]
This allows us to conclude that for every $s \in \mathbb R^3$,  $T f^0_{\rm hom}(s,\cdot)$ is tangentially quasiconvex. 
For the attainment of {\rm (iii)}, we observe that $T f^0_{\textnormal{hom}}(s,\cdot)$ is also rank one convex as long as \eqref{(2.3)} holds since  by {\rm (ii)} $\bar{f}^0_{\rm hom}(s, \cdot)$ is rank one convex.
\color{black}
\\
In view of assumption (H2) and the definition of $T f^0_{\textnormal{hom}}$ it is possible to show that $T f^0_{\rm hom}$ is $p$-coercive and it has uniform $p$-growth in the second variable, uniformly with respect to the first. \color{black}
Since rank one convex functions satisfying uniform $p$-growth and $p$-coercivity conditions are $p$-Lipschitz in view, for instance of \cite[Proposition 2.32]{D}, also \eqref{(2.5)} holds, which concludes the proof of (iii).
\end{proof}

\section{Proof of Theorem \ref{casopm1}}
\label{Sobolev}
%In this section we prove Theorem \ref{casopm1}.
% \begin{thm}
% \label{casopu1}
%     Let $\mathcal{M}$ be a compact and connected smooth submanifold of $\R^3$ without boundary and let $f:\R^3\times\R^{3\times 3}\to [0,+\infty)$ be a Carathéodory function satisfying {\rm (H1)}, {\rm (H2)} with $p=1$, and {\rm (H3)}. Then the $\Gamma-$limit of $I^h$ on $W^{1,1}(\Omega, \mathcal{M})$ as $h \rightarrow 0$ with respect to the strong $L^1$ topology is the functional given by 
% \begin{equation}
%     \label{candidatepu1}
%     I(u) =
% \int_{\omega} T f_{\rm hom}^0(u, \nabla_\alpha u) \, dx, \quad u\in W^{1,1}(\omega,\mathcal{M}),
% \end{equation}
% with $ T f_{\rm hom}^0$ defined as in \eqref{Tfhom}.
% \end{thm}

We recall that $f$ satisfies (H1) and (H2). We also recall that the candidate for the $\Gamma$-limit with respect to the strong $L^p$-topology of the family of functionals $I^h$ in \eqref{functional}, whose localization, for every $p\geq 1$ is defined in $\mathcal A(\omega)$  as  
\[
I^h(u, A) := \left \{
\begin{array}{lll}
\!\!\!\!\!\! & \displaystyle \int_{A_{,1}} f \left (\frac{x_{\alpha}}{h}, x_3, \nabla_h u\right ) \, dx \qquad & \textnormal{if $u \in W^{1,p}(\Omega, \mathcal{M})$}\\[4mm]
\!\!\!\!\!\! & + \infty &\textnormal{elsewhere}
\end{array}
\right. 
\] is the functional $I$ given by \eqref{candidatepm1}.\\

\noindent
For any $A\in\mathcal{A}(\omega)$ and $u\in L^p(\Omega; \mathcal{M})$, consider the functional
%with $\Omega=\omega_{,1}$.
%then it is not difficult to prove that $I^h(u,\cdot)$ is the restriction to $\mathcal{A}(\omega)$ of a Radon measure absolutely continuous with respect to $\mathcal{L}^{2}$. 
%Moreover, we denote
\begin{equation}
\label{Io}
    I^0(u,A):=\inf\left\{\liminf_{n\to+\infty}I^{h_n}(u_n,A) : u_n \to u\,\text{ in $L^p(A_{,1};\mathcal M)$}\right\},
\end{equation}
for any  $u\in W^{1,p}(\omega; \mathcal{M})$ and $A\in\mathcal{A}(\omega)$, and  denote $I^0(u, \omega)$ simply by $I^0(u)$ for every $u\in W^{1,p}(\omega; \mathcal{M})$.\\

\noindent
In order to prove the $\Gamma$- limsup inequality we introduce a suitable functional that is larger than $I^0$ and we prove that it is the restriction to $\mathcal{A}(\omega)$ of a Radon measure absolutely continuous with respect to $\mathcal{L}^{2}$.
Given a compact set $\mathcal{K}\subset\mathcal{M}$, and a non-relabeled subsequence $(h_{n_k})_k:=(h_k)_k$ we define for $u\in W^{1,p}(\omega; \mathcal{M})$ and $A \in \mathcal{A}(\omega)$

 \begin{align*}
% \label{Ikappa}
     I^{\{h_k\}}_\mathcal{K}(u, A)&:=\inf_{(u_k)_k}\Bigg\{ \limsup_{k\to \infty} I^{h_k}(u_k, A)\,:\, \nabla u_k\rightharpoonup \nabla_\alpha u \text{ in } L^{p}(\Omega; \R^3),\\
     &(\nabla_{h_k}u_k)_k \text{ is bounded in } L^p(\Omega,\R^3),\nonumber\\
     &u_k \to u \textnormal{ uniformly, } u_k(x)=u(x_\alpha) \text{ if } {\rm dist}(u(x_\alpha), \mathcal{K})>1 \text{ for a.e. }x\in\Omega
     \Bigg\}. \nonumber
 \end{align*}

\begin{remark}
\label{abuso}
    Given $A\subset\mathcal{A}(\omega)$, the set
\[
V(A) := \left \{v \in W^{1,p}(A_{,1}; \mathcal{M}): \frac{\partial v}{\partial x_3} = 0 \,\,\, \text{a.e. on } A_{,1} \right \},
\]
is isomorphic to the Sobolev space $W^{1,p}(A; \mathcal{M}).$ 
% In the definition \eqref{Ikappa} of $ I^{\{h_n\}}_\mathcal{K}$, the identity $u_k(x)=u(x)$ is to be intended as between $u_k$ and the representative of $u$ in $V(\omega)$, since $u\in W^{1,p}(\omega,\mathcal{M})$ while $u_k\in W^{1,p}(\Omega; \mathcal{M})$.
\end{remark}

%The next results hold true also in the case $p=1$.

\begin{lemma}[Localization]
\label{preupperbound}Let $p\ge1$.
    For every $u\in V(\omega),$ there exists a non-relabeled subsequence $(h_k)_k$ such that the set function $ I^{\{h_k\}}_\mathcal{K}(u, \cdot)$ is the restriction to $\mathcal{A}(\omega)$ of a Radon measure absolutely continuous with respect to $\mathcal{L}^{2}$.
\end{lemma}
 \begin{proof}
From the $p$-growth condition (H2), we obtain that, for any non-relabeled subsequence $(h_k)_k,$
\begin{equation}
\label{pgrowthcons}
 I^{\{h_k\}}_\mathcal{K}(u, A) \le \, c \int_A (1 + |\nabla_\alpha u|^p) \, dx; 
\end{equation}
therefore we just need to infer the existence of a suitable subsequence $(h_k)_k$ for which the trace of $ I^{\{h_k\}}_\mathcal{K}(u, \cdot)$ is a Radon measure. This can be shown in two steps.
\\
\\
{\sc step 1:} The first thing to be proved is that the following subadditivity property for the functional  $I^{\{h_k\}}$ 
\begin{equation}
\label{subadditivity}
 I^{\{h_k\}}_\mathcal{K}(u, A) \le \,  I^{\{h_k\}}_\mathcal{K}(u, B) +  I^{\{h_k\}}_\mathcal{K}(u, A \setminus \overline{C}),
\end{equation}
holds for all $A, B, C \in \mathcal{A}(\omega)$ such that $\overline{C} \subset B \subset A.$ Now, for a given $\eta > 0,$ there exist sequences $(u_k)_k, (v_k)_k \subset W^{1,p}(\Omega; \mathcal{M})$ such that $\nabla u_k$ and $\nabla v_k$ converge weakly to $\nabla_\alpha u$ in $W^{1,p}(\Omega; \mathbb{R}^3) $, and $(\nabla_{h_k} u_k)_k$ and $(\nabla_{h_k} v_k)_k$ are bounded in $L^p(\Omega,\R^3)$, $u_k(x) = v_k(x) = u(x_\alpha)$ if dist$(u(x_\alpha), \mathcal{K}) > 1$ for a.e. $x \in \Omega,$ $u_k$ and $v_k$ are uniformly converging to $u$ and
\[
\left \{
\begin{array}{lll}
 \!\!\! & \displaystyle \limsup_{k \rightarrow + \infty}  I^{h_n}(u_k, B) \le \,  I^{\{h_n\}}_\mathcal{K}(u, B) + \eta\\[2mm]
 \!\!\! & \displaystyle \limsup_{k \rightarrow + \infty}  I^{h_n}(v_k, A \setminus \overline{C}) \le \,  I^{\{h_n\}}_\mathcal{K}(u, A \setminus \overline{C}) + \eta.
 \end{array}
\right.
\]
Now, let us set
\[
\mathcal{K}' := \{s \in \mathcal{M}: {\rm dist} (s, \mathcal{K}) \le \, 1\}.
\]
We observe that $\mathcal{K}'$ is a compact subset of $\mathcal{M}$ and $u_k(x) = v_k(x) = u(x_\alpha)$ if $u(x_\alpha) \notin \mathcal{K}'$ for a.e. $x \in \Omega.$ Moreover, let us fix $L := {\rm dist} (C, \partial B),$ $M \in \mathbb{N}$ and define, for every $i \in \{0, \dots, M\}$
\[
B_i := \left \{
x_\alpha \in B: {\rm dist} (x_\alpha, \partial B) > \frac{i L}{M}
\right \}.
\]
while for every $i \in \{0, \dots, M-1\}$
set 
\[
S_i := B_i \setminus \overline{B_{i+1}}.
\]
Consider finally, for every $i \in \{0, \dots, M-1\}$, $\zeta_i \in \mathcal{C}_c^{\infty}(\Omega; [0,1])$ being a cut-off function satisfying 
\[
\zeta_i(x)=\zeta_i(x_\alpha) = \left \{
\begin{array}{lll}
\!\!\! & 1 \qquad & \textnormal{in $(B_{i + 1})_{,1}$}\\[3mm]
\!\!\! & 0 \qquad & \textnormal{in $\Omega \setminus (B_{i})_{,1}$} \qquad \textnormal{and} \qquad |\nabla \zeta_i|= |\nabla_\alpha \zeta_i|\le \, \frac{2 M}{L}
\end{array}
\right.
\]
By \cite[Lemma 3.2 and Remark 3.3]{Dacorogna-Fonseca-Maly-Trivisa}, there exist $\delta > 0,$ $c > 0$ and a uniformly continuously differentiable mapping $\Phi: D_{\delta} \times [0,1] \rightarrow \mathcal{M},$ where
\[
D_{\delta} := \{(s_0, s_1) \in \mathcal{M} \times \mathcal{M}: {\rm dist}(s_0, \mathcal{K}') < \delta, \,\,\, {\rm dist}(s_1, \mathcal{K}') < \delta, \,\,\, |s_0 - s_1| < \delta\},
\]
such that 
\begin{equation}
\label{(3.4)}
\Phi(s_0, s_1, 0) = s_0, \qquad \Phi(s_0, s_1, 1) = s_1, \qquad \frac{\partial \Phi}{\partial t} (s_0, s_1, t) \le \, c \, |s_0 - s_1|,
\end{equation}
and
\begin{equation}
\label{(3.5)}
|\Phi(s_0, s_1, t) - s_0| \le \, c |s_0 - s_1|.
\end{equation}
We recall that $(u_k)_k$ and $(v_k)_k$ are uniformly converging sequences, thus we can choose $k$ large enough such that the following conditions hold
\[
\|u_k - u\|_{L^{\infty}(\Omega; \mathbb{R}^3)} < \delta, \qquad \|v_k - u\|_{L^{\infty}(\Omega; \mathbb{R}^3)} < \delta.
\]
This entails that, for a.e. $x \in \Omega,$ ${\rm dist}(v_k(x), \mathcal{K}') < \delta$ when $u(x_\alpha) \in \mathcal{K}'.$ This way, we can define $w_{k, i} \in W^{1,p}(\Omega; \mathcal{M})$ as follows
\[
w_{k, i}(x) := \left \{
\begin{array}{lll}
\!\!\! & \Phi(v_k(x), u_k(x), \zeta_i(x)) \qquad & \textnormal{if $u(x) \in \mathcal{K}'$}\\[3mm]
\!\!\! & u(x_\alpha) \qquad & \textnormal{if $u(x) \notin \mathcal{K}'$}
\end{array}
\right.
\]
By exploiting the $p$-growth condition (H2) as well as \eqref{(3.4)}, we are able to deduce
\[
\begin{aligned}
\int_{A_{,1}} &f \left (\frac{x_\alpha}{h_k}, x_3, \nabla_{h_k} w_{k,i} \right ) \, dx \\
&\le \int_{B_{,1}} f \left (\frac{x_\alpha}{h_k}, x_3, \nabla_{h_k} u_k \right ) \, dx + \int_{A_{,1} \setminus \overline{C}_{,1}} f \left (\frac{x_\alpha}{h_k}, x_3, \nabla_{h_k} v_k \right ) \, dx \\
& + C_0 \int_{(S_i)_{,1}} (1 + |\nabla_{h_k} u_k|^p + |\nabla_{h_k} v_k|^p + M^p |u_k - v_k|^p) \, dx,
\end{aligned}
\]
which holds for some constants $C_0 > 0$ independent of $i, k$ and $M.$ Now, if we sum up over the index $i \in \{0, \dots, M-1\}$ and we divide by $M,$ we get
\[
\begin{aligned}
\frac{1}{M} &\sum_{i = 0}^{M - 1} \int_{A_{,1}} f \left (\frac{x_\alpha}{h_k}, x_3, \nabla_{h_k} w_{k,i} \right ) \, dx\\
&\le \int_{B_{,1}} f \left (\frac{x_\alpha}{h_k},  x_3, \nabla_{h_k} u_k \right ) \, dx + \int_{A_{,1} \setminus \overline{C}_{,1}} f \left (\frac{x_\alpha}{h_k}, x_3, \nabla_{h_k} v_k \right ) \, dx \\
& + \frac{C_0}{M} \int_{B_{,1} \setminus \overline{C}_{,1}} (1 + |\nabla_{h_k} u_k|^p + |\nabla_{h_k} v_k|^p + M^p |u_k - v_k|^p) \, dx.
\end{aligned}
\]
Therefore it is possible to find some indices $i_k \in \{0, \dots, M-1\}$ such that $\overline{w}_k := w_{k, i_k}$ satisfies
\[
\begin{aligned}
\int_{A_{,1}} &f \left (\frac{x_\alpha}{h_k},h_k x_3, \nabla_{h_k} \overline{w}_{k} \right ) \, dx\\
&\le\int_{B_{,1}} f \left (\frac{x_\alpha}{h_k}, x_3, \nabla_{h_k} u_k \right ) \, dx + \int_{A_{,1} \setminus \overline{C}_{,1}} f \left (\frac{x_\alpha}{h_k}, x_3,\nabla_{h_k} v_k \right ) \, dx \\
&+ \frac{C_0}{M} \int_{B_{,1} \setminus \overline{C}_{,1}} (1 + |\nabla_{h_k} u_k|^p + |\nabla_{h_k} v_k|^p + M^p |u_k - v_k|^p) \, dx.
\end{aligned}
\]
From \eqref{(3.4)} and \eqref{(3.5)}, we get that $\overline{w}_k \rightarrow u$ uniformly; moreover $\nabla\overline{w}_k \rightharpoonup \nabla_\alpha u$ in $L^p(\Omega; \mathbb{R}^3)$, $(\nabla_{h_k} \bar{w}_k)_k$ is bounded in $L^p(\Omega,\R^3)$, and finally $\overline{w}_k(x) = u(x_\alpha)$ whenever dist$(u(x_\alpha), \mathcal{K}) > 1$ for a.e. $x \in \Omega.$ Finally, we can conclude that
\[
\begin{aligned}
&I_{\mathcal{K}}^{\{h_k\}}(u, A)\\
&\le \limsup_{k \rightarrow + \infty} I^{h_k}(\overline{w}_k, A) \\
&\le \limsup_{k \rightarrow + \infty} \Bigg \{ I^{h_k}(u_k, B) + I^{h_k}(v_k, A \setminus \overline{C}) \\
&\qquad + \frac{C_0}{M} \int_{B_{,1} \setminus \overline{C}_{,1}} (1 + |\nabla_{h_k} u_k|^p + |\nabla_{h_k} v_k|^p + M^p |u_k - v_k|^p) \, dx  \Bigg \}\\
&\le I_{\mathcal{K}}^{\{h_k\}}(u, B) + I_{\mathcal{K}}^{\{h_k\}}(u, A \setminus \overline{C}) + 2 \eta + \frac{C_0}{M} \sup_{k \in \mathbb{N}} \int_{B_{,1} \setminus \overline{C}_1} (1 + |\nabla_{h_k} u_k|^p + |\nabla_{h_k} v_k|^p) \, dx.
\end{aligned}
\]
Taking the limit first as $M \rightarrow + \infty$ and then as $\eta \rightarrow 0,$ we obtain the desired subadditivity property \eqref{subadditivity}.

\noindent {\sc step 2:} At this point, by a standard diagonal argument, we construct a non-relabeled subsequence $h_k \rightarrow 0^+$ and a sequence $(u_k)_k \subset W^{1,p}(\Omega; \mathcal{M})$ satisfying

\begin{align}\nonumber
  \lim_{k \rightarrow + \infty}  & I^{\{h_k\}}(u_k, \omega)=\inf_{v_k}\Bigg\{ \limsup_{k\to \infty} I^{h_k}(v_k, A)\,:\, \nabla v_k\rightharpoonup \nabla_\alpha u \text{ in } L^{p}(\Omega; \R^3),\\
     &(\nabla_{h_k}v_k)_k \text{ is bounded in } L^p(\Omega,\R^3),\nonumber\\
     &v_k \to u \textnormal{ uniformly, } v_k(x)=u(x) \text{ if } {\rm dist}(u(x_\alpha), \mathcal{K})>1 \text{ for a.e. }x\in\Omega
     \Bigg\}. \nonumber
\end{align}

We have that 
\[
\lim_{k \rightarrow + \infty}   I^{{h_k}}(u_k, \omega) = I_{\mathcal{K}}^{\{h_k\}}(u, \omega),
\]
simply by construction of the sequences $(h_k)_k$ and $(u_k)_k.$ By possibly extracting a further subsequence, it is possible to assume that, for some non-negative Radon measure $\mu \in \mathcal{M}(\omega),$
\[
 \left(\int_{(-\frac{1}{2}, \frac{1}{2})}f \left (\frac{(\cdot)_\alpha}{h_k}, x_3, \nabla_{h_k} u_k((\cdot)_\alpha,x_3) \right )dx_3\right) \mathcal{L}^{2} \mres \omega \stackrel{*}{\rightharpoonup} \mu \qquad \textnormal{in} \,\,\,\mathcal{M}(\omega).
% f \left (\frac{(\cdot)_\alpha}{h_k}, (\cdot)_3, \nabla u_k \right ) \mathcal{L}^{3} \mres \Omega \stackrel{*}{\rightharpoonup} \mu \qquad \textnormal{in} \,\,\,\mathcal{M}(\Omega).
\]
We further have, by semicontinuity,
\[
\mu(\omega) \le \, \lim_{k \rightarrow + \infty} I^{\{h_k\}}(u_k, \omega) = I_{\mathcal{K}}^{\{h_k\}}(u, \omega).
\]
We would like to prove that
\[
I_{\mathcal{K}}^{\{h_k\}}(u, A) = \mu(A) \qquad \textnormal{for any $A \in \mathcal{A}(\omega)$}.
\]
Let us fix $A \in \mathcal{A}(\omega)$ and prove that
\[
I_{\mathcal{K}}^{\{h_k\}}(u, A) \le \mu(A) \qquad \textnormal{for any $A \in \mathcal{A}(\omega)$}.
\]
Fix an arbitrary $\eta > 0;$ exploiting \eqref{pgrowthcons}, we can select $C \in \mathcal{A}(\omega),$ $C \Subset A$ such that
\[
I_{\mathcal{K}}^{\{h_k\}}(u, A \setminus \overline{C}) \le \, \eta. 
\]
Then, by \eqref{subadditivity}, we infer that, for any $B \in \mathcal{A}(\omega),$ $C \Subset B \Subset A,$
\[
I_{\mathcal{K}}^{\{h_k\}}(u, A) \le \, \eta + \limsup_{k \rightarrow + \infty} I^{{h_k}}(u_k, B) \le \, \eta + \mu(\overline{B}) \le \, \eta + \mu(A).
\]
By the arbitrariness of $\eta,$ we come to the desired conclusion.
\\
Conversely, for any $B \in \mathcal{A}(\omega),$ $B \Subset A,$ we may deduce
\begin{eqnarray*}
\mu(\omega) & \le & \, I_{\mathcal{K}}^{\{h_k\}}(u, \omega) \\&\le& I_{\mathcal{K}}^{\{h_k\}}(u, A) + I_{\mathcal{K}}^{\{h_k\}}(u, \omega \setminus \overline{B})\\
&\le &\, I_{\mathcal{K}}^{\{h_k\}}(u, A) + \mu(\omega \setminus \overline{B}) \\&\le &\, I_{\mathcal{K}}^{\{h_k\}}(u, A) + \mu(\omega \setminus B) \le \, I_{\mathcal{K}}^{\{h_k\}}(u, A) + \mu(\omega) - \mu(B).
\end{eqnarray*}
This finally entails that
\[
\mu(B) \le \, I_{\mathcal{K}}^{\{h_k\}}(u, A),
\]
which yields the desired conclusion by the inner regularity of $\mu.$
\color{black}   
\end{proof}

\noindent
Now we can prove the $\limsup$-inequality for Theorem \ref{casopm1}.

\begin{prop}[$\Gamma$-limsup]
%\label{GlimsupSobolev}
For every $p\ge 1$ and $u\in W^{1,p}(\omega; \mathcal{M})$ it holds
\begin{equation*}
    I(u)\ge I^0(u),
\end{equation*}
where $I$ and $I_0$ are defined by \eqref{candidatepm1} and \eqref{Io}, respectively.
\end{prop}

\begin{proof}
Let $u \in W^{1,p}(\omega; \mathcal{M}).$ Consider $R > 0$ arbitrarily large, set
\[
\mathcal{K} := \mathcal{M} \cap B^3(0,R),
\]
and consider the sequence $(h_k)_k$ given by Lemma \ref{preupperbound}. It is clear that
\[
I^0(u) \le \, I^{\{h_k\}}_\mathcal{K}(u, \omega).
\]
Now, we would like to show that
\begin{equation}
\label{(4.1)}
I^{\{h_k\}}_\mathcal{K}(u, \omega) \le \int_{\omega} \left \{ \chi_R(|u|) T f_{\rm hom}^0 (u, \nabla_\alpha u) + \beta (1 - \chi_R(|u|)) (1 + |\nabla_\alpha u|^p)\right \} \, dx_\alpha,
\end{equation}
where 
\[
\chi_R(t) = 
\left \{
\begin{array}{lll}
\!\!\! 1 \qquad \textnormal{for $t \le R$}\\
\\
\!\!\! 0 \qquad \textnormal{otherwise}
\end{array}
\right.
\]
In order to deduce \eqref{(4.1)}, it is enough to prove that
\[
\frac{d I^{\{h_k\}}_\mathcal{K}(u, \cdot) }{d \mathcal{L}^2} (x_0) \le \, \chi_R(|u(x_0)|) T f_{\rm hom}^0 (u(x_0), \nabla u(x_0)) + \beta (1 - \chi_R(|u(x_0)|) (1 + |\nabla u(x_0)|^p),
\]
for $\mathcal{L}^2-$a.e. $x_0 \in \omega.$
\\
Let us consider $x_0 \in \omega$ to be a Lebesgue point of $u$ and $\nabla_\alpha u$ such that $u(x_0) \in \mathcal{M},$ $\nabla_\alpha u(x_0) \in [T_{u(x_0)}(\mathcal{M})]^2,$ and the Radon-Nykod\'ym derivative of $I^{\{h_k\}}_\mathcal{K}(u, \cdot)$ with respect to the Lebesgue measure $\mathcal{L}^2$ exists. We observe that almost every point in $\omega$ satisfies these properties. Moreover let us set 
\[
s_0 := u(x_0) \qquad \textnormal{and} \qquad \xi_0 := \nabla_\alpha u(x_0).
\]
Assume first that $s_0 \notin \mathcal{K}.$ Then, using (H2), we obtain that
\begin{eqnarray}
\label{caso1}
\frac{d I^{\{h_k\}}_\mathcal{K}(u, \cdot) }{d \mathcal{L}^2} (x_0) &= &\lim_{\rho \rightarrow 0^+} \frac{I^{\{h_k\}}_\mathcal{K}(u, Q'(x_0, \rho))}{\rho^2} \le \, \limsup_{\rho \rightarrow 0^+} \limsup_{k \rightarrow + \infty} \rho^{-2} I^{h_k}(u, Q'(x_0, \rho))\nonumber \\
&\le & \, \lim_{\rho \rightarrow 0^+} \frac{\beta}{\rho^2} \int_{Q'(x_0, \rho)} (1 + |\nabla_\alpha u|^p) \, dx = \beta (1 + |\xi_0|^p),
\end{eqnarray}
which is what we would like to prove.\\

\noindent
If instead $s_0 \in \mathcal{K},$ then, fixed an arbitrary $0 < \eta < 1,$ Proposition \ref{characterization}-(i) entails the existence of $j \in \mathbb{N}$ and $\varphi \in W^{1, \infty}((jQ')_{,1}; T_{s_0}(\mathcal{M}))$  such that $\varphi(x_{\alpha}, x_3) = 0$ for every $(x_{\alpha}, x_3) \in \partial (jQ') \times (0,1)$ and such that

\begin{equation}
\label{(4.2)}
\media_{(jQ')_{,1}} f(y, \xi_0 + \nabla_\alpha \varphi(y)| \nabla_3\varphi(y)) \, d y \le \, T f_{\rm hom}^0 (s_0, \xi_0) + \eta.
\end{equation}
At this point, let us extend $\varphi((\cdot)_\alpha, x_3)$ to $\mathbb{R}^2$ by $j$-periodicity and define $\varphi_k(x) := \xi_0 x_\alpha + h_k \varphi\left (\frac{x_\alpha}{h_k}, x_3 \right ).$
\\
Consider $\mathcal{U}$ to be an open neighborhood of $\mathcal{M}$ such that the nearest point projection $\Pi: \mathcal{U} \rightarrow \mathcal{M}$ defines a $\mathcal{C}^1-$mapping; fix $\sigma, \delta_0 \in (0,1)$ such that $B^3(s_0, 2 \delta_0) \subset \mathcal{U}$ and consider $\delta = \delta(\sigma) \in (0, \delta_0)$ such that
\begin{equation}
\label{(4.3)}
|\nabla \Pi(s) - \nabla \Pi(s')| < \sigma \qquad \textnormal{for all $s, s' \in B^3(s_0, \delta_0)$ satisfying $|s - s'| < \delta$}.
\end{equation}
Introduce the cut-off function $\zeta \in \mathcal{C}^{\infty}_c(\mathbb{R}^3; [0,1])$ as
\[
\zeta(X) = 
\left \{
\begin{array}{lll}
\!\!\! 1 \qquad \textnormal{for $X \in B^3(0, \delta/4)$}\\
\!\!\! 0\qquad \textnormal{for $X \notin B^3(0, \delta/2)$}
\end{array}
\right. \qquad \textnormal{with} \qquad |\nabla \zeta| \le \frac{C}{\delta},
\]
and define
\[
w_k(x) := u(x_\alpha) + h_k \zeta (u(x_\alpha) - s_0) \varphi \left ( \frac{x_\alpha}{h_k},x_3\right ).
\]
We recall that the function $u$ has to be intended in the sense of  Remark \ref{abuso}, so $w_k$ is well defined as a function of $(x_\alpha,x_3)$.
%and $u$ do not share the same domain, therefore all the operations between $w_k$ and $u$ must be intended in the sense of  
Let $k_0 \in \mathbb{N}$ be such that
\begin{equation}
\label{(4.4)}
\max \left \{h_k \|\varphi\|_{L^{\infty}((jQ')_{,1}; \mathbb{R}^3)} \|\nabla \zeta\|_{L^{\infty}(\mathbb{R}^3; \mathbb{R}^3)}, \frac{2 h_k \|\varphi\|_{L^{\infty}((jQ')_{,1}; \mathbb{R}^3)}}{\delta} \right \} < 1 ,
\end{equation}
for any $k \ge k_0$ and define, still for every $k \ge k_0,$
\[
u_k(x) := \Pi (w_k(x)).
\]
By \eqref{(4.4)}, for a.e. $x \in \Omega$ and all $k \ge k_0,$ we deduce that $w_k(x) \in B^3(s_0, \delta)$ when $|u(x_\alpha) - s_0| < \delta/2$, while $w_k(x) = u(x_\alpha)$ when $|u(x_\alpha) - s_0| \ge \, \delta/2.$ Thus $u_k$ is well defined, $(u_k)_k \subset W^{1,p}(\Omega; \mathcal{M}),$ and, for a.e.\,$x \in \Omega,$ $u_k(x) = u(x_\alpha)$ when dist$(u(x_\alpha), \mathcal{K}) > 1.$
\\
In addition to that
\begin{eqnarray*}
\|u_k - u\|_{L^{\infty}(\Omega; \mathbb{R}^3)} &=& \|\Pi(w_k) - \Pi(u)\|_{L^{\infty}(\{|u - s_0| < \delta/2\}_{,1}; \mathbb{R}^3)}\\
&\le &  \, h_k\|\nabla \Pi\|_{L^{\infty}(B^3(s_0, \delta_0); \mathbb{R}^3)} \|\varphi\|_{L^{\infty}((jQ')_{,1}; \mathbb{R}^3)} \rightarrow 0 \qquad \textnormal{as $k \rightarrow + \infty.$}
\end{eqnarray*}
By the Chain Rule formula, we can compute
\begin{eqnarray*}
\nabla_{h_k} u_k(x) &=& \nabla\Pi(w_k(x)) \Bigg [\nabla_\alpha u(x_\alpha) + h_k \left (\varphi \left (\frac{x_\alpha}{h_k},x_3 \right ) \otimes \nabla \zeta (u(x_\alpha) - s_0) \right ) \nabla_\alpha u(x_\alpha)\\
&& + h_k\zeta (u(x_\alpha) - s_0) \nabla_{h_k} \left(\varphi \left (\frac{x_\alpha}{h_k}, x_3 \right ) \right)\Bigg ],
\end{eqnarray*}
which, in turn, entails
\begin{eqnarray*}
&& |\nabla_{h_k} u_k(x)| \nonumber \\
&\le& \, \|\nabla \Pi\|_{L^{\infty}(B^3(s_0, \delta_0); \mathbb{R}^3)} \Bigg [ \left (1 + h_k \|\varphi\|_{L^{\infty}((jQ')_{,1}; \mathbb{R}^3)} \|\nabla \zeta\|_{L^{\infty} (\mathbb{R}^3; \mathbb{R}^3)}\right ) |\nabla_\alpha u(x_\alpha)|\\
&& + \|\nabla \varphi\|_{L^{\infty}((jQ')_{,1}; \mathbb{R}^3)}\Bigg].
\end{eqnarray*}
Once more \eqref{(4.4)} entails that, for any $k \ge k_0$ and for some constant $C_0$ depending on $s_0, \xi_0, \delta_0, \eta$ and independent of $x$ and $k,$
\begin{equation}
\label{(4.5)}
|\nabla_{h_k} u_k(x)| \le \, C_0 (|\nabla_\alpha u(x_\alpha) - \xi_0| + 1).
\end{equation}

\noindent Thus, we can conclude that the sequence $(u_k)_k$ is uniformly bounded in $W^{1,p}(\Omega; \mathbb{R}^3)$ so that $\nabla u_k \rightharpoonup \nabla_\alpha u$ in $L^p(\Omega; \mathbb{R}^3).$ Moreover, we observe that $|\nabla_{h_k} u_k| \le \, 2 C_0$ a.e. in $\{|\nabla_\alpha u - \xi_0| < \sigma\}_{,1}$ and
\begin{align*}
    \|\nabla_{h_k} \varphi_k\|_{L^{\infty}(\Omega; \mathbb{R}^{3 \times 2})} &\le \, |\xi_0| +\|\nabla \varphi\|_{L^{\infty}((jQ')_{,1}; \mathbb{R}^3)}
\end{align*}
Set 
\begin{equation}
\label{(4.6)}
M := \max \{2 C_0, |\xi_0| +\|\nabla \varphi\|_{L^{\infty}((jQ')_{,1}; \mathbb{R}^3)}\},
\end{equation}
depending only on $s_0, \xi_0, \delta_0$ and $\eta$, so that
\begin{equation}
\label{(4.7)}
|\nabla_{h_k} u_k| \le \, M \qquad \textnormal{and} \qquad |\nabla_{h_k} \varphi_k| \le \, M \qquad \textnormal{a.e. in $\{|\nabla_\alpha u - \xi_0| < \sigma\}_{,1}.$}
\end{equation}
At this point, for a.e. 
\[
x \in \{|u - s_0| < \delta/4\}_{,1} \cap \{|\nabla_\alpha u - \xi_0| < \sigma\}_{,1},
\]
we have $\zeta (u(x) - s_0) = 1$ and
\begin{eqnarray*}
|\nabla_{h_k} u_k(x) - \nabla_{h_k} \varphi_k(x)| &\le & \, |\nabla \Pi(w_k) \nabla_\alpha u(x_\alpha) - \xi_0| \\
&& + |h_k\nabla \Pi(w_k) \nabla_{h_k} (\varphi(x_\alpha/h_k,x_3)) - h_k\nabla_{h_k} (\varphi(x_\alpha/h_k,x_3))|\\
&\le& |\nabla \Pi(w_k) - \nabla \Pi(s_0)| \, |\nabla_\alpha u(x_\alpha)| + |\nabla \Pi(s_0)| |\nabla_\alpha u(x_\alpha) - \xi_0|\\
&&+ |\nabla \Pi(w_k) - \nabla \Pi(s_0)| \, \|\nabla \varphi\|_{L^{\infty}((jQ')_{,1}; \mathbb{R}^3)};
\end{eqnarray*}
where, in the last inequality, we have used the fact that, since $\nabla \varphi(y) \in [T_{s_0} (\mathcal{M})]^3$ for a.e. $y \in \mathbb{R}^3,$ it hold $\nabla \Pi(s_0) \nabla \varphi(y) = \nabla \varphi(y)$ and $\nabla \Pi(s_0) \xi_0 = \xi_0.$
\\
Now, taking into account \eqref{(4.3)} and the fact that $|w_k - s_0| < \delta$ a.e. in $\{|u - s_0| < \delta/4\}_{,1} \cap \{|\nabla u - \xi_0| < \sigma\}_{,1},$ we deduce
\begin{equation}
\label{(4.8)}
|\nabla_{h_k} u_k(x) - \nabla_{h_k} \varphi_k(x)| \le \, (|\nabla_\alpha u(x_\alpha)| + |\nabla \Pi(s_0)| + \|\nabla \varphi\|_{L^{\infty}((jQ')_{,1}; \mathbb{R}^3)}) \sigma \le \, C_1 \sigma,
\end{equation}
a.e. in $\{|u - s_0| < \delta/4\}_{,1} \cap \{|\nabla_\alpha u - \xi_0| < \sigma\}_{,1},$ where $C_1 = C_1 (s_0, \xi_0, \delta_0, \eta) > 0$ is a constant independent of $\sigma, k$ and $x.$
\\
Now we estimate

\begin{align}
&\frac{d I^{\{h_k\}}_\mathcal{K}(u, \cdot) }{d \mathcal{L}^2} (x_0)\nonumber\\
&=\lim_{\rho \rightarrow 0^+} \frac{I^{\{h_k\}}_\mathcal{K}(u, Q'(x_0, \rho))}{\rho^2} \nonumber \\
&\le \limsup_{\rho \rightarrow 0^+} \limsup_{k \rightarrow + \infty} \frac{1}{\rho^2} \int_{Q'(x_0, \rho)_{,1}} f\left (\frac{x_\alpha}{h_k}, x_3, \nabla_{h_k} u_k \right ) \, dx \nonumber\\
&\le \limsup_{\rho \rightarrow 0^+} \limsup_{k \rightarrow + \infty} \frac{1}{\rho^2} \int_{Q'(x_0, \rho)_{,1} \cap \{|u - s_0| \ge \, \delta/4\}_{,1}} f\left (\frac{x_\alpha}{h_k}, x_3, \nabla_{h_k} u_k \right ) \, dx \nonumber\\
& \quad +\limsup_{\rho \rightarrow 0^+} \limsup_{k \rightarrow + \infty} \frac{1}{\rho^2} \int_{Q'(x_0, \rho)_{,1} \cap \{|u - s_0| < \delta/4\}_{,1} \cap \{|\nabla u - \xi_0| < \sigma\}_{,1}} f\left (\frac{x_\alpha}{h_k}, x_3, \nabla_{h_k} u_k \right ) \, dx \nonumber\\
& \quad + \limsup_{\rho \rightarrow 0^+} \limsup_{k \rightarrow + \infty} \frac{1}{\rho^2} \int_{Q'(x_0, \rho)_{,1} \cap \{|u - s_0| < \delta/4\}_{,1} \cap \{|\nabla u - \xi_0| \ge \sigma\}_1} f\left (\frac{x_\alpha}{h_k}, x_3, \nabla_{h_k} u_k \right ) \, dx \nonumber\\
&=: I_1 + I_2 + I_3. \label{(4.9)}
\end{align}
The bound \eqref{(4.5)}, the $p$-growth condition (H2) and the selected choice of $x_0$ yield
\begin{eqnarray*}
I_1 &\le &\, C \, \limsup_{\rho \rightarrow 0^+} \frac{1}{\rho^2} \int_{Q'(x_0, \rho) \cap \{|u - s_0| \ge \, \delta/4\}} (1 + |\nabla_\alpha u(x_\alpha) - \xi_0|^p) \, dx\nonumber \\
&\le & C \, \limsup_{\rho \rightarrow 0^+} \media_{Q'(x_0, \rho)} |\nabla_\alpha u(x_\alpha) - \xi_0|^p \, dx + \frac{4C}{\sigma} \limsup_{\rho \rightarrow 0^+} \media_{Q'(x_0, \rho)} |u(x_\alpha) - s_0| \, dx = 0, %\label{stimaI1}
\end{eqnarray*}
while
\begin{eqnarray}
I_3 &\le &\, C \, \limsup_{\rho \rightarrow 0^+} \frac{1}{\rho^2} \int_{Q'(x_0, \rho) \cap \{|u - s_0| < \, \delta/4\} \cap \{|\nabla_\alpha u - \xi_0| \ge \, \sigma\}} (1 + |\nabla_\alpha u(x_\alpha) - \xi_0|^p) \, dx \nonumber\\
&\le & C \, \limsup_{\rho \rightarrow 0^+} \media_{Q'(x_0, \rho)} |\nabla_\alpha u(x_\alpha) - \xi_0|^p \, dx\nonumber \\
&& + \frac{C}{\sigma} \limsup_{\rho \rightarrow 0^+} \media_{Q'(x_0, \rho)} |\nabla_\alpha u(x_\alpha) - \xi_0| \, dx = 0,  \label{(4.11)}
\end{eqnarray}
Concerning integral $I_2,$ since, for a.e. $y \in \mathbb{R}^3,$ the function $f(y, \cdot)$ is continuous, it is uniformly continuous on $B^{3 \times 3}(0, M)$ where $M > 0$ has been introduced in \eqref{(4.6)}.
\\
Define
\[
\omega(y, t) := \sup \{|f(y, \xi) - f(y, \xi')|: \xi, \xi' \in B^{3 \times 3}(0, M) \,\,\, \textnormal{and} \,\,\, |\xi - \xi'| \le t\},
\]
to be the modulus of continuity of $f(y, \cdot)$ over $B^{3 \times 3}(0, M).$ It is not difficult to see that $\omega(y, \cdot)$ is increasing, continuous and $\omega(y, 0) = 0;$ on the other hand $\omega(\cdot, t)$ is measurable, since the supremum can be restricted to all admissible $\xi$ and $\xi'$ having rational entries, and 1-periodic.
\\
Accounting on \eqref{(4.7)} and \eqref{(4.8)}, we are able to estimate for a.e. $x \in Q'(x_0, \rho)_{,1} \cap \{|u - s_0| < \delta/4\}_{,1} \cap \{|\nabla_\alpha u - \xi_0| < \sigma\}_{,1}$
\[
\left | f \left ( \frac{x_\alpha}{h_k}, x_3, \nabla_{h_k} u_k(x)\right ) -  f \left ( \frac{x_\alpha}{h_k}, x_3, \nabla_{h_k} \varphi_k(x)\right )\right | \le \, \omega \left (\frac{x_\alpha}{h_k}, x_3, C_{1} \sigma \right ).
\]
Integrating over $x \in Q'(x_0, \rho)_{,1} \cap \{|u - s_0| < \delta/4\}_{,1} \cap \{|\nabla_\alpha u - \xi_0| < \sigma\}_{,1}$ and taking the limit as $k \rightarrow + \infty,$ we obtain
\begin{eqnarray*}
&& \limsup_{k \rightarrow + \infty} \frac{1}{\rho^2} \int_{ Q'(x_0, \rho)_{,1} \cap \{|u - s_0| < \delta/4\}_{,1} \cap \{|\nabla_\alpha u - \xi_0| < \sigma\}_{,1}} \Bigg| f \left ( \frac{x_\alpha}{h_k}, x_3, \nabla_{h_k} u_k(x)\right ) \\
&& -  f \left ( \frac{x_\alpha}{h_k}, x_3,\nabla_{h_k} \varphi_k(x)\right )\Bigg | \, dx \\
&\le & \limsup_{k \rightarrow + \infty} \frac{1}{\rho^2} \int_{Q'(x_0, \rho)_{,1}} \omega \left (\frac{x_\alpha}{h_k}, x_3, C_{1} \sigma \right ) \, dx \\
&=& \int_{Q} \omega(y, C_{1} \sigma) \, dy,
\end{eqnarray*}
where we have exploited the Riemann-Lebesgue Lemma and where we used the fact that $y \mapsto \omega(y, C_{1} \sigma)$ is a measurable 1-periodic function. 
\\
The Dominated Convergence Theorem and the fact that $\omega(y, 0) = 0$ for every $y\in\R^3$ yield 
\begin{equation}
\label{(4.12)}
\lim_{\sigma \rightarrow 0^+} \int_{Q} \omega(y, C_{1} \sigma) \, dy = 0.
\end{equation}
Summing up, 
\begin{equation}
\label{(4.13)}
I_2 \le \,  \limsup_{\rho \rightarrow 0^+}  \limsup_{k \rightarrow + \infty} \frac{1}{\rho^2} \int_{Q'(x_0, \rho)_{,1}} f \left ( \frac{x_\alpha}{h_k}, x_3, \nabla_{h_k} \varphi_k(x)\right ) \, dx + \int_{Q'} \omega(y, C_{1} \sigma) \, dy.
\end{equation}
By the definition of $\varphi_k$ and once more the Riemann-Lebesgue Lemma, we are able to deduce from \eqref{(4.2)} that
\begin{eqnarray}
 &&\limsup_{\rho \rightarrow 0^+}  \limsup_{k \rightarrow + \infty} \frac{1}{\rho^2} \int_{Q'(x_0, \rho)_{,1}} f \left ( \frac{x_\alpha}{h_k}, x_3, \nabla_{h_k} \varphi_k \left (x \right)\right ) \, dx\nonumber\\
 &= &\media_{(jQ')_{,1}} f(y,\xi_0 + \nabla_\alpha\varphi(y)| \nabla_3\varphi(y)) \, dy \nonumber\\
 &\le & \, T f_{\rm hom}^0 (s_0, \xi_0) + \eta.  \label{(4.14)}
\end{eqnarray}
Finally, collecting \eqref{(4.9)}--\eqref{(4.11)}, \eqref{(4.13)} and also \eqref{(4.14)}, we are able to obtain
\[
\frac{d I^{\{h_k\}}_\mathcal{K}(u, \cdot) }{d \mathcal{L}^2} (x_0) \le \, T f_{\rm hom}^0 (s_0, \xi_0) + \int_Q \omega(y, C_{1} \sigma) \, dy + \eta.
\]
The thesis now follows by sending first $\sigma \rightarrow 0$ (exploiting \eqref{(4.12)}) and then $\eta \rightarrow 0.$
\\
Once \eqref{(4.1)} is obtained, we can conclude by considering a sequence $R_j \rightarrow + \infty$ as $j \rightarrow + \infty.$ By \eqref{caso1} and \eqref{(4.14)}, since $\chi_{R_j} \rightarrow 1$ pointwise, we can deduce from Dominated Convergence Theorem together with the $p$-growth of $Tf^0_{\rm hom}$(see Proposition \ref{characterization} (iii)) that
\begin{eqnarray*}
I^0(u) &\le &\limsup_{j \rightarrow + \infty} \int_{\Omega} \left \{\chi_{R_j}(|u|) T f_{\rm hom}^0 (u, \nabla_\alpha u) + \beta (1 - \chi_{R_j}(|u|)) (1 + |\nabla_\alpha u|^p) \right \} \, dx \\
&\le & \, \int_{\Omega} T f_{\rm hom}^0(u, \nabla_\alpha u) \, dx,
\end{eqnarray*}
which is what we wanted to prove.
    \color{black}
\end{proof}

\noindent
Now we prove the $\liminf$-inequality for Theorem \ref{casopm1}, i.e. for $p>1$.

\begin{prop}[$\Gamma$-liminf]
    For every $p> 1$ and $u\in W^{1,p}(\omega; \mathcal{M})$ it holds
    \begin{equation*}
        I(u)\le I^0(u),
    \end{equation*}
    where $I$ is defined in \eqref{candidatepm1}, while $I^0$ is defined in \eqref{Io}.
\end{prop}
\begin{proof}
Before proving the result we recall that all the operations of sum and difference between the functions $u_n, v_{n,k}, v_k, w_k$ and $u$ must be intended in the sense of Remark \ref{abuso}.\\

\noindent
{\sc step 1.} Fix $u\in W^{1,p}(\omega;\mathcal{M})$. We consider a recovery sequence $(u_n)_n\subset W^{1,p}(\Omega; \mathcal{M})$ related to $I^0(u,\omega)$. We define the sequence of non-negative Radon measure
    \[
\mu_n:=\left(\int_{(-\frac{1}{2}, \frac{1}{2})}f\left(\frac{(\cdot)_\alpha}{h_n}, x_3, \nabla_{h_n}u_n\right)dx_3\right)\mathcal{L}^{2}\mres\omega.
    % \mu_n:=f\left(\frac{(\cdot)_\alpha}{h_n}, (\cdot)_3, \nabla_{h_n}u_n\right)\mathcal{L}^{3}\mres\Omega.
    \]
    Up to a subsequence, there exists a Radon measure $\mu\in \mathcal{M}(\omega)$ such that $\mu_n\rightharpoonup^*\mu$ \\
    in $\mathcal{M}(\omega)$. 
    By Lebesgue Differentiation Theorem we can split $\mu$ into the sum of two mutually disjoint non-negative Radon measure $\mu^a$ and $\mu^s$. In particular, $\mu^a<<\mathcal{L}^{2}$, while $\mu^s$ is singular with respect to $\mathcal{L}^{2}$. By definition, $\mu^a(\Omega)\le\mu(\Omega)\le I^0(u)$ so we want to prove that
    \[
    \frac{d\mu}{d\mathcal{L}^{2}}(y_0)\ge Tf^0_{\rm hom}(u(y_0),\nabla_\alpha u(y_0)) \quad \text{for }\mathcal{L}^{2}-\text{ a.e. } y_0\in\omega.
    \]

    \noindent
    Let $y_0\in\omega$ be a Lebesgue point for $u$ and $\nabla_\alpha u$ and a point of of approximate differentiability for $u$, i.e. such that $u(y_0)\in \mathcal{M}$ and $\nabla_\alpha u(y_0)\in [T_{u(y_0)}(\mathcal{M})]^{2}$, and such that the Radon-Nykod\'ym derivative of $\mu$ with respect to $\mathcal{L}^{2}$ exists and it is finite. We define $s_0:=u(y_0)$ and $\xi_0:=\nabla_\alpha u(y_0)$ and we consider a vanishing sequence $(\rho_k)_k\subset (0,+\infty)$ such that $\mu(\partial Q' (y_0,\rho_k))=0$ for every $k\in\N$. By definition \eqref{perturbedf} of $\Bar{f}$ we get
    \begin{eqnarray*}
        \frac{d\mu}{d\mathcal{L}^{2}}(y_0)&=&\lim_{k\to +\infty}\frac{\mu(Q'(y_0,\rho_k))}{\rho_k^{2}}\nonumber\\
        &=&\lim_{k\to +\infty}\,\lim_{n\to +\infty} \frac{\mu_n(Q' (y_0,\rho_k)_{1})}{\rho_k^{2}}\nonumber\\
        &=&\lim_{k\to +\infty}\,\lim_{n\to +\infty}\int_{Q'_{,1}}f\left(\frac{y_0+\rho_ky_\alpha}{h_n},y_3, \nabla_{h_n}u_n(y_0+\rho_ky_\alpha, y_3)\right)dy\nonumber\\
        &=&\lim_{k\to +\infty}\,\lim_{n\to +\infty}\int_{Q'_{,1}} \Bar{f}\left(\frac{y_0+\rho_ky_\alpha}{h_n},y_3, u_n(y_0+\rho_ky_\alpha, y_3), \nabla_{h_n}u_n(y_0+\rho_ky_\alpha, y_3)\right)dy\nonumber\\
        &=&\lim_{k\to +\infty}\,\lim_{n\to +\infty}\int_{Q'_{,1}} \Bar{f}\left(\frac{y_0+\rho_ky_\alpha}{h_n},y_3, s_0+v_{n,k}(y), \nabla_{\frac{h_k}{\rho_k}}v_{n,k}(y)\right)dy,\nonumber
    \end{eqnarray*}
    with $v_{n,k}(y):=\frac{\left(u_n(y_0+\rho_ky_\alpha, y_3)-s_0\right)}{\rho_k}.$ 
\\
%%%%%%%%%%%%%%%%%%%%%%%%%%%%5
Since $y_0$ is a point of approximate differentiability  and $u_n\to u$ in $L^p(\Omega, \R^3)$ it follows that
\begin{equation*}
    \lim_{k\to \infty} \lim_{n\to \infty} \int_{Q'_{,1}} |v_{n,k}(y)-\xi_0y_\alpha|^pdy=\lim_{k\to \infty}\int_{Q'(y_0, \rho_k)}\frac{|u(y)-s_0-\xi_0(y_\alpha-y_0)|^p}{\rho^{2+p}}dy_\alpha=0.
\end{equation*}

Therefore it is possible to find a diagonal sequence $h_k := h_{n_k} < \rho_k^2$ such that, by setting $v_k(y) := v_{n_k, k} (y)$ with $y\in\Omega$, $v_0(y_\alpha) := \xi_0 y_\alpha$ with $y_\alpha\in\omega$, then $v_k \rightarrow v_0$ in $L^p(Q'_1; \mathbb{R}^{3})$ and
\begin{equation}
\label{(5.1)}
\frac{d \mu}{d \mathcal{L}^{2}}(y_0) = \lim_{k \rightarrow + \infty} \int_{Q'_{,1}} \Bar{f} \left (\frac{y_0 + \rho_k y_\alpha}{h_k},y_3, s_0 + \rho_k v_k(y), \nabla_{\frac{h_k}{\rho_k}} v_k(y) \right ) \, dy.
\end{equation}
At this point, we observe that $(\nabla_{h_k} v_k)_k$ is bounded in $L^p(Q'_{,1}; \mathbb{R}^{3\times3})$ thanks to the coercivity condition (H2). By using the Decomposition Theorem \cite[Theorem 1.1]{BF} (see also \cite{BZ}), it is possible to find a sequence $(\bar{v}_k)_k \subset W^{1, \infty}(Q'_{,1}; \mathbb{R}^{3})$ such that $\bar{v}_k = v_0$ on a neighborhood of $\partial (Q')_{,1},$ $\bar{v}_k \rightarrow v_0$ in $L^p(Q'_1; \mathbb{R}^{3}),$ the sequence of gradients $(|\nabla_{h_k} \bar{v}_k|^p)_k$ is equi-integrable, \color{black} and
\begin{eqnarray}
&& \lim_{k \rightarrow + \infty} \int_{Q'_{,1}} \Bar{f} \left (\frac{y_0 + \rho_k y_\alpha}{h_k},y_3, s_0 + \rho_k v_k(y), \nabla_{\frac{h_k}{\rho_k}}  v_k(y) \right ) \, dy \nonumber \\
&\ge& \, \limsup_{k \rightarrow + \infty} \int_{Q'_{,1}} \Bar{f} \left (\frac{y_0 + \rho_k y_\alpha}{h_k},y_3, s_0 + \rho_k v_k(y), \nabla_{\frac{h_k}{\rho_k}} \bar{v}_k(y) \right ) \, dy. \label{(5.2)}
\end{eqnarray}

\vspace{5mm}

\noindent{\sc step 2.} Let us set
\[
\frac{y_0}{h_k} = m_k + s_k \qquad \textnormal{with} \,\,\, m_k \in \mathbb{Z}^{2} \,\,\, \textnormal{and} \,\,\, s_k \in [0,1)^{2}.
\]
We can introduce
\[
x_k := \frac{h_k}{\rho_k} s_k \rightarrow 0 \qquad \textnormal{and} \qquad \delta_k := \frac{h_k}{\rho_k} \rightarrow 0.
\]
We can exploit the 1-periodicity of $\Bar{f}$ with respect to its first variable, \eqref{(5.1)} and \eqref{(5.2)} to get
\begin{eqnarray}
\frac{d \mu}{d \mathcal{L}^{2}}(x_0) &\ge & \, \limsup_{k \rightarrow + \infty} \int_{Q'_{,1}}\Bar{f} \left (\frac{x_k + y_\alpha}{\delta_k},y_3, s_0 + \rho_k v_k(y), \nabla_{\delta_k}  \bar{v}_k(y) \right ) \, dy \label{(5.3)} \\
&\ge &\limsup_{k \rightarrow + \infty} \int_{x_k +Q'_{,1}}\Bar{f} \left (\frac{y_\alpha}{\delta_k},y_3, s_0 + \rho_k v_k(y_\alpha - x_k, y_3), \nabla_{\delta_k} \bar{v}_k(y_\alpha - x_k, y_3) \right ) \, dy. \nonumber 
\end{eqnarray}
At this point, we extend $v_k$ as its limit (up to fixing $v_k$ at the boundary of $\partial \omega \times (-1,1)$) and $\bar{v}_k$ by $v_0$ to  $\mathbb{R}^{2}\times (-1,1).$ \color{black} As long as $x_k \rightarrow 0,$ we deduce that $\mathcal{L}^{3} ((Q'_{,1} - x_k) \Delta Q'_{,1}) \rightarrow 0,$ and the equi-integrability of $(|\nabla_{h_k} \bar{v}_k|^p)_k$ together with the $p$-growth condition (H2) implies
\begin{eqnarray*}
&& \int_{Q'_{,1} \Delta (x_k + Q'_{,1})} \Bar{f} \left (\frac{y_\alpha}{\delta_k},y_3, s_0 + \rho_k v_k(y_\alpha - x_k, y_3), \nabla_{\delta_k} \bar{v}_k(y_\alpha - x_k,y_3) \right ) \, dy \\
&\le & \, \beta' \int_{Q'_{,1} \Delta (Q'_{,1})-x_k} (1 + |\nabla_{\delta_k} \bar{v}_k|^p ) \, dy \rightarrow 0.
\end{eqnarray*}
Therefore \eqref{(5.3)} entails 
\[
\frac{d \mu}{d \mathcal{L}^{2}} (y_0) \ge \, \limsup_{k \rightarrow + \infty} \int_{Q'_{,1}} \Bar{f} \left (\frac{y_\alpha}{\delta_k},y_3, s_0 + \rho_k w_k, \nabla_{\delta_k} \bar{w}_k \right ) \, dy,
\]
where $w_k(y) := v_k (y_\alpha - x_k, y_3)$ and $\bar{w}_k(y) := \bar{v}_k(y_\alpha - x_k, y_3)$ converge to $v_0$ in $L^p(Q'_{,1}; \mathbb{R}^{3}),$ and $(|\nabla \bar{w}_k|^p)_k$ is equi-integrable as well.
%\comment{Andrea: Lo step 3 consiste nel bloccare la variabile $u$ per poi poter applicare Braides-Fonseca-Francfort. Che io sappia non ci sono altri risultati in letteratura che si possono applicare. C'è una connessione con i lavori di Neukamm? Controllare, anche chi lo ha citato. Probabilmente non serve niente perché Neukamm viene dopo BFF e quindi lui sa che nel caso membranale libero c'è già tutto in BFF}
\vspace{5mm}

\noindent{\sc step 3.} 
Fixed $M>0$, we denote by $E_{M,k}$ the set
\[
E_{M,k}:=\left\{ x\in Q'_{,1}: |\nabla_{h_k} w_k|\le M \right\}.
\]
By Chebyschev inequality, we have that $\mathcal{L}^3(Q'_1\setminus E_{M,k})\le \frac{C}{M^p}$ for some constant $C>0$. By Scorza-Dragoni Theorem, fixed $\eta>0$ there exists a compact set $K_\eta \subset \overline{Q'_{,1}}$ such that $\mathcal{L}^3(\overline{Q'_{,1}}\setminus K_\eta)\le \eta$ and such that $f:K_\eta\times\R^{3\times 3}\to \R$ is continuous. It follows that $\Bar{f}(\cdot, s, \cdot):K_\eta\times B^{3\times 3}(0,M)\to\R$ is uniformly continuous for every $s\in\R^3$. Moreover, the function
\[
\Psi_{\eta, M}(t):=\sup \left\{ |f(x,\xi)-f(x, \zeta)|: x\in K_\eta,\, \xi,\zeta\in B^{3\times 3}(0,M), |\xi-\zeta|\le t \right\},
\]
is continuous, takes value $0$ for $t=0$ and is bounded. By definition of $\Psi_{\eta, M}$ and $\mathbf P_{s}$ and \cite[Proposition 2.32]{D} it follows that for every $x\in K_\eta$, $\xi\in B^{3\times 3}(0,M)$ and $s_1, s_2\in\R^3$ holds
\begin{align}
    |\Bar{f}(x,s_1, \xi)-\Bar{f}(x,s_2, \xi)|&\le \Psi_{\eta,M}(|\mathbf P_{s_1}(\xi)-\mathbf P_{s_2}(\xi)|) + C_M|\mathbf P_{s_1}(\xi)-\mathbf P_{s_2}(\xi)| \nonumber\\
    & \le \Psi_{\eta,M}(M|\mathbf P_{s_1}- \mathbf P{s_2}|) + C_M|\mathbf P_{s_1}- \mathbf P_{s_2}|\nonumber\\
    &:=\tilde\Psi(|\mathbf P_{s_1}- \mathbf P_{s_2}|), \nonumber
\end{align}
where $|\mathbf P_{s_1}- \mathbf P{s_2}|$ denotes the operator norm of $\mathbf P_{s_1}- \mathbf P_{s_2}$. By the previous inequality, it follows that if we denote
\[
K_\eta^{per}:= \bigcup_{j\in\Z}(j+K_\eta),
\]
then
\[
 |\Bar{f}(x,s_1, \xi)-\Bar{f}(x,s_2, \xi)|\le\tilde\Psi(|\mathbf P_{s_1}- \mathbf P{s_2}|), \nonumber
\]
for every  $x\in K_\eta^{per}$, $\xi\in B^{3\times 3}(0,M)$ and $s_1, s_2\in\R^3$.
From the previous inequality it follows that
\begin{align}
    \frac{d \mu}{d \mathcal{L}^{2}} (y_0)&\ge \limsup_{k \rightarrow + \infty} I^{h_k}(\bar{w}_k, Q'_1)\nonumber\\
    & \ge \limsup_{k \rightarrow + \infty} \int_{E_{M,k}\cap (\delta_k K_\eta^{per})}  \Bar{f}\left (\frac{y_\alpha}{\delta_k},y_3, s_0, \nabla_{\delta_k} \bar{w}_{k} \right ) dy\nonumber\\
    & - \limsup_{k \rightarrow + \infty} C_M\int_{Q'_{,1}} \tilde\Psi_{\eta, M}(|\mathbf P_{s_0+\rho_k w_k(y)}- \mathbf P_{s_0}|) dy. \nonumber
\end{align}
Since $\tilde\Psi_{\eta, M}$ is continuous, bounded and $\tilde\Psi_{\eta, M}(0)=0$ and since $|\mathbf P_{s_0+\rho_k w_k(y)}- \mathbf P_{s_0}| \to 0$ as $k\to \infty$, then the last term in the previous inequality is also $0$. It follows that
\[
 \frac{d \mu}{d \mathcal{L}^{2}} (y_0) \ge 
 \limsup_{k \rightarrow +\infty} \int_{E_{M,k}\cap \delta_k K_\eta^{per}} \bar f\left(\frac{y_\alpha}{\delta_k}, y_3,s_0,\nabla_{\delta_k} \bar w_{k} \right )dy.
\]
From the $p$-growth of $\bar f$ and from Riemann-Lebesgue Lemma we get that
\begin{align}
    &\limsup_{k \rightarrow + \infty} \int_{E_{M,k}\setminus \delta_k K_\eta^{per}}  \bar f\left(\frac{y_\alpha}{\delta_k},y_3, s_0, \nabla_{\delta_k} \bar w_{k} \right)dy\nonumber\\
    & \le \limsup_{k \rightarrow + \infty} C(1+M^p)\mathcal{L}^3(Q'_1\setminus \delta_k K_\eta^{per})\nonumber\\
    & = C(1+M^p)\mathcal{L}^3(Q'_1\setminus K_\eta)\nonumber\\
    & \le C(1+M^p)\eta.\nonumber
\end{align}
From the previous inequality we deduce that
\[
\frac{d \mu}{d \mathcal{L}^{2}} (y_0) \ge \limsup_{k \rightarrow + \infty} \int_{E_{M,k}}  \Bar{f}\left (\frac{y_\alpha}{\delta_k},y_3, s_0, \nabla_{\delta_k} \bar w_{k} \right )dy - C(1+M^p)\eta.
\]
Since $\eta$ is arbitrary, then for $\eta \to 0$ we get
\begin{align}
    \label{lowerboundM}
    \frac{d \mu}{d \mathcal{L}^{2}} (y_0) \ge \limsup_{k \rightarrow + \infty} \int_{E_{M,k}}  \Bar{f}\left (\frac{y_\alpha}{\delta_k},y_3, s_0, \nabla_{\delta_k} \bar w_{k} \right )dy.
\end{align}
On the other hand, by construction, $\mathcal{L}^3(Q'_{,1}\setminus E_{M,k})\to 0$ uniformly with respect to $k$ as $M\to +\infty$. Since $(|\nabla_{h_k}\bar w_k|^p)_k$ is equi-integrable in $\Omega$, then from the $p$-growth of $\bar f$ we get for $M \to +\infty$
\[
\sup_{k}\int_{Q'_{,1}\setminus E_{M,k}} \bar f\left(\frac{y_\alpha}{\delta_k},y_3, s_0,\nabla_{\delta_k}\bar w_k\right)dy \le \sup_{k} C \int_{Q'_{,1}\setminus E_{M,k}} \left( 1+ |\nabla_{\delta_k}\bar w_k|^p\right)dy \to 0. 
\]
From this limit and from \eqref{lowerboundM} we conclude that
\[
\frac{d \mu}{d \mathcal{L}^{2}} (y_0) \ge \limsup_{k \rightarrow + \infty} \int_{Q'_{,1}}  \bar f\left (\frac{y_\alpha}{\delta_k},y_3, s_0, \nabla_{\delta_k} \bar w_{k} \right )dy.
\]
Using now \cite[Theorem 4.2]{BFF} we get
\[
\frac{d \mu}{d \mathcal{L}^{2}} (y_0)\ge \int_{Q'} \bar f_{\rm hom}^0(u(y_0), \nabla_\alpha u(y_0))dy= \bar f_{\rm hom}^0(u(y_0), \nabla_\alpha u(y_0)).
\]
By Proposition \ref{characterization} it follows that
\[
\bar f_{\rm hom}^0(u(y_0), \nabla_\alpha u(y_0))=Tf_{\rm hom}^0(u(y_0), \nabla_\alpha u(y_0)).
\]

\end{proof}

\section*{Acknowledgment}
The authors gratefully acknowledge support from INdAM GNAMPA.
	A.T.\,have been partially supported through the INdAM-GNAMPA 2025 project ``Minimal surfaces: the Plateau problem and behind'' cup E5324001950001. The research was in part carried out within the project: Geometric-Analytic Methods for PDEs and Applications (GAMPA) , ref. 2022SLTHCE – cup E53D2300588 0006 - funded by European Union - Next Generation EU within the PRIN 2022 program (D.D. 104 - 02/02/2022 Ministero dell’Università e della Ricerca). This manuscript reflects only the authors’ views and opinions and the Ministry cannot be considered responsible for them. 
    E. Z. has been supported by PRIN 2022: Mathematical Modelling of Heterogeneous Systems (MMHS)
- Next Generation EU CUP B53D23009360006 and by INdAM GNAMPA Project 2024 'Composite materials and microstructures'.


\begin{thebibliography}{99}

%\bibitem{Alberti} {\sc G. Alberti:} 
%\emph{Rank-one property for derivatives of functions with bounded variation.} 
%Proc. R. Soc. Edinburgh
%Sect. A 123 (1993), pp. 239-274.

\bibitem{ABP91} {\sc E. Acerbi, G. Buttazzo, D. Percivale:} 
\emph{A variational definition of the strain energy for an elastic string.} 
Journal of Elasticity, 25, (1991), 137-148.

%\bibitem{AEL07} {\sc R. Alicandro, A. Corbo Esposito, C. Leone:} 
%\emph{Relaxation in BV of integral functionals defined on Sobolev functions with values in the unit sphere,} 
%J. Conv. Anal. 14 (2007), pp. 69–98. 

\bibitem{AL01} {\sc R. Alicandro, C. Leone:} 
\emph{3D-2D asymptotic analysis for micromagnetic energies.} 
ESAIM: COCV, 6 (2001), 489–498.

\bibitem{All92}{\sc G. Allaire :} \emph{Homogenization and two-scale convergence.} SIAM Journal on Mathematical Analysis, 23 (6) (1992), 1482-1518.

%\bibitem{All95} {\sc G. Allair: }
%\emph{Relaxation of structural optimization problems by homogenization.} Trends in Applications of Mathematics to Mechanics”, MM Marques and JF Rodrigues Eds., Pitman monographs and surveys in pure and applied mathematics, 1995, 77: 237-251.

%\bibitem{AB}{\sc L. Ambrosio, A. Braides: } \emph{ Functionals defined on partitions in sets of finite perimeter. I. Integral representation and G-convergence.} Journal de Mathématiques Pures et Appliquées, 69 (1990), 285-305.

\bibitem{Ambrosio-Dal Maso}{\sc L. Ambrosio, G. Dal Maso:} 
\emph{On the relaxation in $BV(\Omega, \R^m)$ of quasiconvex integrals}. 
J. Funct. Anal., 109 (1992), 76–97.

%\bibitem{AFP} {\sc L. Ambrosio, N. Fusco, D: Pallara:} {Functions of Bounded Variation and Free Discontinuity Problems.}
%Oxford University Press, New York, 2000.

\bibitem{BMbv} {\sc J.-F. Babadjian, V. Millot:} \emph{Homogenization of variational problems in manifold valued $BV$-spaces.} 
Calc. Var., 36, (2009) 7–47.

\bibitem{BM} {\sc J.-F. Babadjian, V. Millot:} \emph{Homogenization of variational problems in manifold valued Sobolev Spaces}, 
{ESAIM: COCV} 16 (2010), 833–855.

\bibitem{BZ07} {\sc M. Ba\'ia, E. Zappale:} 
     \emph{A note on the 3{D}-2{D} dimensional reduction of a
              micromagnetic thin film with nonhomogeneous profile},b{Appl. Anal.}, 86 (5), (2007),  555--575.
              
\bibitem{BLP78}
{\sc A. Bensoussan, J. L. Lions, G. Papanicolau: }
Asymptotic analysis for periodic structures (Vol. 5). 
Elsevier, (1978).

%\bibitem{Bet91}{\sc F. Bethuel:}\emph{The approximation problem for Sobolev maps between two manifolds}, Acta Math., 167 (1991), 153-206.

%\bibitem{BZ88}{\sc F. Bethuel, X. Zheng: } \emph{Density of smooth functions between two manifolds in Sobolev spaces}. Journal of functional analysis, 80(1) (1988), 60-75.

\bibitem{BF}{\sc M. Bocea, I. Fonseca:} \emph{Equi-integrability results for 3D-2D dimension reduction problems}. ESAIM: Control, Optimisation and Calculus of Variations, 2002, 7: 443-470.

\bibitem{BFM} {\sc G. Bouchitt\'e, I. Fonseca, L. Mascarenhas:} 
\emph{A Global Method for Relaxation}, 
Arch. Rational Mech. Anal. 145 (1998), pp. 51-98.

\bibitem{BDF} {\sc A. Braides, A. Defranceschi:} {Homogenization of multiple integrals}, 
{\it Oxford Lecture Series in Mathematics and its Applications} 12. Oxford University Press, New York, 1998.

%\bibitem{BDFV} {\sc A. Braides, A. Defranceschi, A. Vitali: }
%\emph{Homogenization of free discontinuity problems.} Arch. Rat. Mech. Anal. 135 (1996), pp. 297–356 

\bibitem{BFF} {\sc A. Braides, I. Fonseca, G. Francfort: }\emph{3D-2D Asymptotic Analysis for Inhomogeneous Thin Films}, 
Indiana University Mathematics Journal, 49, (4) (2000), 1367-1404.

\bibitem{BZ}{\sc A. Braides, C. I. Zeppieri:} \emph{A note on equi-integrability in dimension reduction problems}. Calc. Var. (2007) 29:231-238

%\bibitem{Buttazzolibro} {\sc G. Buttazzo:} {Semicontinuity, relaxation, and integral representation in the calculus of variations.}
%Longman Group, 1989.


\bibitem{BV17} {\sc M. Bukal, I. Vel\v ci\'c:} \emph{On the simultaneous homogenization and dimension reduction in elasticity and locality of $\Gamma$-closure}, Calc. Var. Partial Differential Equations, 56 (3) (2017),  41 pp.

\bibitem{CCG} {\sc L. Carbone, K. Chacouche, A. Gaudiello: }
     \emph{Fin junction of ferroelectric thin films},
{Adv. Calc. Var.}, 11 (4), (2018), 341--371.

\bibitem{CRZ11}{\sc  G. Carita, A. M. Ribeiro, E. Zappale}: \emph{ An Homogenization Result in $W^{1,p}\times L^q$.} Journal of Convex Analysis, 18, No. 4, (2011), 1093-1126.

\bibitem{CGO24} {\sc A. Chakrabortty, G. Griso, J. Orlik}, \emph{Dimension reduction and homogenization of composite plate with matrix pre-strain}, Asymptot. Anal., 138 (4), (2024), 255--310.

\bibitem{CDG02} {\sc D. Cioranescu, A. Damlamian, G. Grisio: }
\emph{Periodic unfolding and homogenization.}
Comptes rendus. Mathématique,  335 (1), (2002), 99-104.

\bibitem{CDGbook} {\sc D. Cioranescu, A. Damlamian, G. Griso:}
{\it The periodic unfolding method}, Ser.\,Contemp.\,Math.\,3, (2018), 513 pp.
 %Theory and Applications to Partial Differential Problems 

\bibitem{DMbook}{\sc G. Dal Maso:} 
{An Introduction to $\Gamma$-convergence.}
Progress in Nonlinear Differential Equations and Their Applications 
Volume 8, 1993.

\bibitem{D} {\sc B. Dacorogna:} {Direct methods in the calculus of variations}. Springer-Verlag, 2008.

\bibitem{Dacorogna-Fonseca-Maly-Trivisa} {\sc B. Dacorogna, I. Fonseca, J. Mal\'y, K. Trivisa:}
\emph{Manifold constrained variational problems}. 
Calc. Var. Part. Diff. Eq., 9 (1999), 185–206.

\bibitem{Dam05} {\sc A. Damlamian:} 
\emph{An elementary introduction to periodic unfolding.} 
 Multi scale problems and asymptotic analysis, 24,  (2005), 119-136.
 
\bibitem{DKMO}
{\sc A. Desimone, R. Kohn, S. Müller and F. Otto:} 
\emph{A reduced theory for thin-film micromagnetics.} Comm. Pure Appl. Math., 55 (2002), 1408–1460.

\bibitem{LTZ} {\sc L. Lussardi, A. Torricelli, E. Zappale:} \emph{Homogenization and 3D-2D dimension reduction of a functional on manifold valued $BV$ space}, (2025), submitted.

%\bibitem{FM} {\sc I. Fonseca, S. M\"uller:} 
%\emph{Quasiconvex integrands and lower semicontinuity in $L^1$}, 
%{\it SIAM J. Math. Anal.,} 23 (1992), pp. 1081-1098.

%\bibitem{FM93} {\sc I. Fonseca, S. M\"uller:} 
%\emph{Relaxation of quasiconvex functionals in $BV(\Omega; \mathbb{R}^p)$ for integrands $f(x, u, \nabla u),$ }
%{\it Arch. Rat. Mech. Anal.,} 123 (1993), pp. 1-49.

\bibitem{FG23} {\sc J. Fabricius, M. Gahn:} \emph{Homogenization and dimension reduction of the Stokes problem with Navier-slip condition in thin perforated layers}, Multiscale Model. Simul. 21 (4), (2023), 1502--1533.

\bibitem{FIOM14} {\sc T. Fatima, E. Ijioma, T. Ogawa, A. Muntean:} \emph{Homogenization and dimension reduction of filtration combustion in heterogeneous thin layers}, Netw. Heterog. Media 9 (4) (2014),  709-737.


% \bibitem{FMP98} {\sc I. Fonseca, S. M\"uller and P.\,Pedregal:} 
% \emph{Analysis of concentration and oscillation effects generated by gradients,} 
% {SIAM J. Math. Anal.}, 29, (1998), 736-756.

\bibitem{FMT09} {\sc G. Francfort, F. Murat, L. Tartar:}\emph{Homogenization of monotone operators in divergence form with x-dependent multivalued graphs.} Annali di Matematica Pura ed Applicata, 188, (2009), 631-652.

\bibitem{GH} {\sc A. Gaudiello, R. Hadiji: }
      \emph{Ferromagnetic thin multi-structures},
   {J. Differential Equations},
 257 (5), (2014), 1591--1622.

 \bibitem{GH2} {\sc A. Gaudiello, R. Hadiji: } \emph{Junction of ferromagnetic thin films}
Calc. Var. Partial Differential Equations 39 (3-4), (2010), 593–619.

\bibitem{GJ} {\sc G. Gioia, R. D. James: }
 \emph{Micromagnetics of very thin films. }
 Proceedings of the Royal Society of London. Series A: Mathematical, Physical and Engineering Sciences,  453 (1956), (1997), 213-223.

%\bibitem{KR} {\sc J. Kristensen, F. Rindler:} \emph{Piecewise affine approximations for functions
%of bounded variation}. Numer. Math. 132 (2016), pp. 329-346

\bibitem{KK16} {\sc C. Kreisbeck, S. Kr\"omer:} \emph{Heterogeneous thin films: combining homogenization and dimension reduction with directors}, SIAM J. Math. Anal., 48 (2), (2016),  785-820.

\bibitem{LDR95} {\sc H. Le Dret, A. Raoult: }
 \emph{The nonlinear membrane model as variational limit of nonlinear three-dimensional elasticity. }
Journal de mathématiques pures et appliquées,  74 (6), (1995), 549-578.

\bibitem{P} {\sc G. Pisante: }\emph{Homogenization of micromagnetics large bodies},
{ESAIM Control Optim. Calc. Var.}, 10 (2), (2004), 295--314.
\bibitem{Neukamm} {\sc S.M. Neukamm:} \emph{Homogenization, linearization and dimension reduction in elasticity with variational methods}, Ph.D. Thesis, Technische Universit\"at M\"unchen, Zentrum Mathematik, 2010.

\bibitem{Ngu89} {\sc G. Nguetseng: }
\emph{A general convergence result for a functional related to the theory of homogenization.}
SIAM Journal on Mathematical Analysis, 20 (3), (1989), 608-623.

\bibitem{SP80}{\sc E. Sanchez-Palencia:}
\emph{Fluid flow in porous media.} 
Non-homogeneous media and vibration theory, (1980), 129-157.

\bibitem{Tar77} {\sc L. Tartar:} 
\emph{Homogénéisation et compacité par compensation.}
Séminaire Équations aux dérivées partielles (Polytechnique) dit aussi" Séminaire Goulaouic-Schwartz", (1977), 1-12.

\bibitem{Tar09} {\sc L. Tartar:} 
{The general theory of homogenization: a personalized introduction.} Springer Science \& Business Media, 2009.

\bibitem{Vis06} {\sc A. Visintin:} 
 \emph{Towards a two-scale calculus. }
 ESAIM: Control, Optimisation and Calculus of Variations, 12 (3), (2006),  371-397.

 \bibitem{Vis07} {\sc A. Visintin: }
 \emph{Two-scale convergence of some integral functionals.} 
Calculus of Variations and Partial Differential Equations, 29, (2007), 239-265.

\end{thebibliography}
\end{document}